\documentclass[a4paper,11pt, oneside]{amsart}
\usepackage{amsmath}
\usepackage{amsthm}
\usepackage{amssymb}
\usepackage[margin=0.5in]{caption}
\usepackage{comment}
\usepackage{dsfont}
\usepackage{enumitem}
\usepackage{graphicx}
\usepackage{mathrsfs}
\usepackage{mathtools}
\usepackage{nameref}
\usepackage[new]{old-arrows}
\usepackage{seqsplit}
\usepackage{setspace}
\usepackage{stmaryrd}
\usepackage{thmtools}
\usepackage{xurl}
\usepackage{xcolor}
\usepackage{xparse}
\usepackage{pdfpages} 

\usepackage{tikz}
\usetikzlibrary{cd}
\usetikzlibrary{positioning, angles, quotes}

\theoremstyle{plain}
\newtheorem{thm}{Theorem}[section]      \newtheorem*{thm*}{Theorem}

\newtheorem{cor}[thm]{Corollary}        \newtheorem*{cor*}{Corollary}

\newtheorem{prop}[thm]{Proposition}     \newtheorem*{prop*}{Proposition}
            \newtheorem*{lem*}{Lemma}

\newtheorem{claim}[thm]{Claim}          \newtheorem*{claim*}{Claim}
        \newtheorem*{exer*}{Exercise}
\newtheorem{q}[thm]{Question}           \newtheorem*{q*}{Question}
\newtheorem{prob}[thm]{Problem}           \newtheorem*{prob*}{Problem}
\newtheorem{conj}[thm]{Conjecture}      \newtheorem*{conj*}{Conjecture}

\theoremstyle{definition}
\newtheorem{defn}[thm]{Definition}      \newtheorem*{defn*}{Definition}

           \newtheorem*{ex*}{Example}
    \newtheorem*{notation*}{Notation}
          \newtheorem*{setup*}{Setup}
		\newtheorem*{const*}{Construction}

\theoremstyle{remark}
\newtheorem{rem}[thm]{Remark}           \newtheorem*{rem*}{Remark}
      \newtheorem*{conv*}{Conventions}

\newtheoremstyle{iremark}
    {0.5em\topsep}   
    {0.5em\topsep}   
    {\upshape}  
    {0pt}       
    {\itshape}  
    {.}         
    {5pt plus 1pt minus 1pt} 
    {\thmname{#1}\thmnumber{ \itshape#2}\thmnote{ (#3)}} 
\theoremstyle{iremark}

\theoremstyle{plain}

\usepackage[hidelinks]{hyperref}

\usepackage[capitalise]{cleveref}  

\Crefname{claim}{Claim}{Claims}
\Crefname{conj}{Conjecture}{Conjectures}
\Crefname{const}{Construction}{Constructions}
\Crefname{cor}{Corollary}{Corollaries}
\Crefname{defn}{Definition}{Definitions}
\Crefname{ex}{Example}{Examples}
\Crefname{lem}{Lemma}{Lemmas}
\Crefname{prop}{Proposition}{Propositions}
\Crefname{q}{Question}{Questions}
\Crefname{prob}{Problem}{Problems}
\Crefname{rem}{Remark}{Remarks}
\Crefname{thm}{Theorem}{Theorems}
\Crefname{manualtheoreminner}{Theorem}{Theorems}

\newcommand{\VN}{\mathcal{N}}
\newcommand{\UO}{\mathcal{U}}

\newcommand{\C}{\mathbb{C}}

\newcommand{\Q}{\mathbb{Q}}
\newcommand{\R}{\mathbb{R}}
\newcommand{\Z}{\mathbb{Z}}

\DeclareMathOperator{\cd}{cd}

\DeclareMathOperator{\Ext}{Ext}

\DeclareMathOperator{\Hom}{Hom}
\DeclareMathOperator{\id}{id}

\DeclareMathOperator{\im}{im}

\DeclareMathOperator{\Mat}{Mat}

\DeclareMathOperator{\Tor}{Tor}

\newcommand{\vertii}[1]{{\left\vert\kern-0.25ex\left\vert #1 \right\vert\kern-0.25ex\right\vert}}
\newcommand{\vertiii}[1]{{\left\vert\kern-0.25ex\left\vert\kern-0.25ex\left\vert #1 \right\vert\kern-0.25ex\right\vert\kern-0.25ex\right\vert}}

\makeatletter
\newsavebox{\@brx}
\newcommand{\llangle}[1][]{\savebox{\@brx}{\(\m@th{#1\langle}\)}%
  \mathopen{\copy\@brx\mkern2mu\kern-0.9\wd\@brx\usebox{\@brx}}}
\newcommand{\rrangle}[1][]{\savebox{\@brx}{\(\m@th{#1\rangle}\)}%
  \mathclose{\copy\@brx\mkern2mu\kern-0.9\wd\@brx\usebox{\@brx}}}
\makeatother

\newcounter{comments}

\title[Simon's knot genus problem and Lewin \texorpdfstring{$3$}{3}-manifold groups]{Simon's knot genus problem and Lewin \texorpdfstring{$3$}{3}-manifold groups}

\author{Pablo S\'{a}nchez-Peralta}
\address[P. S\'{a}nchez-Peralta]{Emmanuel College, University of Cambridge, St Andrew's Street, Cambridge CB2~3AP, United Kingdom}
\email{ps2089@cam.ac.uk}

\begin{document}

\begin{abstract}
	We provide a positive answer to an old problem of Jonathan K. Simon: if $K$ and $K'$ are two knots such that there is an epimorphism from the knot group of $K$ to the knot group of $K'$, then the genus of $K$ is greater than or equal to the genus of $K'$. This result follows from our proof of a conjecture of Friedl and L\"uck. Specifically, we show that any map between admissible $3$-manifolds that induces an epimorphism on the fundamental groups and an isomorphism on the rational homologies yields an inequality of Thurston norms. Under these conditions, we also establish new upper bounds on the number of faces of the Thurston norm ball of the target manifold.
	
	To resolve Friedl and L\"uck's conjecture, we prove that locally indicable $3$-manifold groups are Lewin groups, thereby confirming a conjecture of Jaikin-Zapirain within the class of $3$-manifold groups. As a further consequence of our methods, we show that the crossed product of a division ring and a torsion-free $3$-manifold group that is virtually free-by-cyclic is a pseudo-Sylvester domain. 
\end{abstract}

\maketitle

\section{Introduction} \label{sec: introd}

Let $K$ be a knot in the $3$-sphere $S^3$ and let $X_K$ be the exterior $S^3 \setminus \upsilon(K)$ of $K$, where $\upsilon(K)$ is an open tubular neighborhood of $K$. We write $G_K$ for the knot group of $K$, that is, the fundamental group $\pi_1(X_K)$. For two knots $K$ and $K'$, if there exists a group epimorphism $G_K \twoheadrightarrow G_{K'}$, then it is natural to expect that $K$ is more complicated than $K'$ in some sense. For instance, it was shown in \cite{GordonLuecke_Knots} that if $K$ and $K'$ are prime knots and the epimorphism is an actual isomorphism, then $K$ and $K'$ are equivalentent, meaning that there is a homeomorphism $h \colon S^3 \to S^3$ such that $h(K) = K'$. Moreover, Hempel proved in \cite{Hempel_RF3mflds} that knot groups are residually finite, and hence Hopfian, so (in the prime case) $G_{K'}$ cannot surject on $G_K$ unless $K$ and $K'$ are equivalent. However, if $n(G)$ denotes the minimum number of (meridian) generators, the inequality $n(G_K) > n(G_{K'})$ proposed in \cite[Problem 1.12(A)]{Kirby_Problems} does not hold in general as the group of the torus knot $(3p,2)$, $p$ odd, maps onto the trefoil knot group \cite{HartleyMurasugi_Meridians}.  In Kirby's problem list \cite{Kirby_Problems}, Jonathan K. Simon posed the following seminal problem in the 1970s concerning the genus of knots, denoted as $g(K)$.

\begin{prob}[{\cite[Problem 1.12(B)]{Kirby_Problems}}] \label{prob: Simon}
	Let $K$ and $K'$ be two knots. If there is a group epimorphism from the knot group $G_K$ to the knot group $G_{K'}$, does this imply that $g(K)$ is greater than or equal to $g(K')$?
\end{prob}

As pointed out in \cite[p. 278]{Kirby_Problems}, it follows from classical results that the inequality holds if the degree of the Alexander polynomial of $K'$ is twice the genus of $K'$; in particular, the inequality holds if $K'$ is fibered (though this forces the Alexander polynomial to be monic). In \cite[Corollary 6.22]{Gabai_Foliations} and \cite[Theorem 8.8]{Gabai_FoliationsIII}, Gabai gave an affirmative answer when the epimorphism is induced by a proper map of non-zero degree between the knot exteriors. In \cite{BoileauKitanoNozaki_Genera}, Boileau, Kitano and Nozaki provided a positive answer for a symmetric union and its partial knot. Central to our work are the connections established by Friedl and L\"uck in \cite{FriedlLuck_euler} between $L^2$-Betti numbers and the Thurston norm, which in turn extends the genus of a knot. With this framework, they solved \cref{prob: Simon} when the knot group $G_{K'}$ is residually-\{locally indicable and elementary amenable\} (which includes the case when $K'$ is fibered \cite[Lemma 7.5]{FriedlLuck_euler}). Let us recall the definition of the Thurston norm as in \cite{Thurston_Norm}.

Given a surface $F$ with connected components $F_1, \ldots, F_n$, we define 
\[
\chi_{-}(F) := \sum_{i = 1}^n \max\{ -\chi(F_i), 0\}.
\]
Let $M$ be a connected, orientable, compact $3$-manifold $M$ with empty or non-empty boundary, and let $\phi \in H^1(M; \Z) = \Hom(\pi_1(M), \Z)$. The \emph{Thurston norm} $x_M(\phi)$ is defined as
\[
x_M(\phi) := \min \{ \chi_{-}(F) : \mbox{ $F \subset M$ properly embedded surface dual to $\phi$} \}.
\]
If $K \subseteq S^3$ is a knot, we have $H^1(X_K; \Z) \cong \Z$, and a routinary exercise shows that for any generator $\phi \in H^1(X_K; \Z)$ it holds
\[
	x_{X_K}(\phi) = \max \{ 2g(K) - 1, 0\}.
\]
Thurston \cite{Thurston_Norm} showed that this function is homogeneous and subadditive, that is, for $\phi, \phi_1, \phi_2 \in H^1(M; \Z)$ and $n \in \Z$ it holds that
\[
x_M(n \phi) = \lvert n \rvert x_M(\phi) \quad \mbox{and} \quad x_M(\phi_1+\phi_2) \leq x_M(\phi_1) + x_M(\phi_2).
\]
These properties imply that  $x_M$  extends to a function on $H^{1}(M; \mathbb{Q})$ which can then be extended by continuity to a seminorm on $H^{1}(M; \mathbb{R})$ which we denote by $x_M$ again.

Following Friedl and L\"uck (see \cite[Definition 0.1]{FriedlLuck_euler}), we say that a $3$-manifold $M$ is \emph{admissible} if $M$ is connected, orientable, compact and irreducible, its boundary is empty or a disjoint union of tori, and its fundamental group $\pi_1(M)$ is inﬁnite.  In this paper we prove the following conjecture of Friedl and L\"uck \cite[Conjecture 0.5]{FriedlLuck_euler}.

\begin{thm} \label{thm: ineq_thurston3}
	Let $f \colon M \to N$ be a map between admissible $3$-manifolds which induces an epimorphism on the fundamental groups and an isomorphism $f_{*} \colon H_n(M; \Q) \to H_n(N; \Q)$ for every non-negative integer $n$. Then for every $\phi \in H^1(N; \R)$ it holds that
	\[
	 	 x_M(f^{*} \phi ) \geq x_N(\phi)
	\]
	where $f^{*} \phi$ is the pullback of $\phi$ with $f$.
\end{thm} 

We record that the rational homology condition of \cref{thm: ineq_thurston3} cannot be dropped as shown by Liu's example in \cite[Remark 0.6 (3)]{FriedlLuck_euler}, which exhibits a map $f$ from $M = S^1 \times \Sigma$ with $\Sigma$ a surface of genus $g \geq 2$ with boundary, to $N$ the exterior of a non-trivial torus knot, such that the pullback $f^{*} \phi$ of a generator $\phi$ of $H^1(N;\Z)$ is dual to the ``vertical tori'' in $S^1 \times \Sigma$. As a corollary to \cref{thm: ineq_thurston3}, we give a positive answer to \cref{prob: Simon} in full generality.

\begin{cor} \label{cor: Simon}
	Let $K$ and $K'$ be two knots. If $G_K$ surjects onto $G_{K'}$, then $g(K) \geq g(K')$.
\end{cor}

In their breakthrough paper, Agol and Liu \cite{AgolLiu_Simon} showed that every knot group maps onto at most finitely many knot groups resolving \cite[Problems 1.12(C) \& (D)]{Kirby_Problems} conjectured by Simon. \cref{cor: Simon} settles part (B), thereby fully resolving Problem 1.12 from Kirby's problem list \cite{Kirby_Problems}. Observe that one can define a preorder on knots where $K \geq K'$ if there is an epimorphism from $G_K$ to $G_{K'}$. It follows then that, for this preorder, each knot has finitely many descendants all of which have smaller or equal genus.

For a compact and orientable $3$-manifold $M$ the \emph{Thurston norm ball} is defined by
\[
	B_{x_M} :=\{ \phi \in H^1(M;\R) \colon x_M(\phi) \leq 1 \}
\]
and the \emph{dual Thurston polytope} is given by
\[
T(M)^{*} := B_{x_M}^{*} = \{ p \in H_1(M; \R) \colon \phi(p) \leq x_M(\phi) \mbox{ for all } \phi \in H^1(M; \R) \}.
\]
Thurston \cite[Theorem 2, p. 106 and first paragraph, p. 107]{Thurston_Norm} showed that $T(M)^{*}$ is integral, that is, the vertices lie in $H_1(M)_{\text{f}} \subseteq H_1(M;\R)$. Define $F_{x_M}$ as the number of integral points contained in $T(M)^{*}$. As a byproduct of \cref{thm: ineq_thurston3} we obtain that $F_{x_M}$ is an upper for the number of top-dimensional faces of the Thurston norm ball of the target manifold.

\begin{cor} \label{cor: Thurston_ball}
	Let $M$ be an admissible $3$-manifold. Then for every admissible $3$-manifold $N$ admitting a map $f \colon M \to N$ which induces an epimorphism on the fundamental groups and an isomorphism $f_{*} \colon H_n(M; \Q) \to H_n(N; \Q)$ for every non-negative integer $n$, the number of top-dimensional faces of $B_{x_N}$ is bounded by $F_{x_M}$.
\end{cor} 

Now we give some remarks on the proof of \cref{thm: ineq_thurston3}. In \cite{FriedlLuck_euler}, Friedl and L\"uck showed that for any admissible $3$-manifold $N$ and any surjective $\phi \in H^1(N; \Z)$
\[
	x_N(\phi) = b_1^{(2)}(\ker \phi).
\]
This reduces \cref{thm: ineq_thurston3} to showing that $b_1^{(2)}(\ker \phi) \leq b_1^{(2)}(\ker f^{*}\phi)$. Friedl and L\"uck proved the latter inequality whenever the fundamental group of the target $3$-manifold is residually-\{locally indicable and elementary amenable\} (see \cite[Theorem 7.3]{FriedlLuck_euler} for the precise statement). For a general admissible $3$-manifold $N$, let $G$ denote $\pi_1(N)$, if $\phi$ is non-trivial we only know that $G$ is locally indicable (see \cref{cor: admissible_Lewin}). In this case, the group algebra $\Q[G]$ embeds into the Linnell division ring $\mathcal{D}_{\Q[G]}$ by \cite[Theorem 3.2(3)]{FriedlLuck_euler} (see \cref{sec: L2Betti} for more details). Thus, given an arbitrary finitely presented $\Q[G]$-module $L$, we have two different notions of dimension of $L$, namely
\[
	\dim_{\Q} \left( \Q \otimes_{\Q[G]} L \right) \quad \mbox{and} \quad \dim_{\mathcal{D}_{\Q[G]}} \left( \mathcal{D}_{\Q[G]} \otimes_{\Q[G]} L \right)
\]
where $\dim_{\mathcal{D}_{\Q[G]}}$ is just the usual rank of a module over a division ring. In this module theoretic language, Friedl and L\"uck showed that when $G$ is in addition residually-\{locally indicable and elementary amenable\}, then for every finitely generated free $\Q[G]$-complex $C_{*}$
\begin{equation} \label{eqtn: ineq_intro}
	\dim_{\mathcal{D}_{\Q[G]}} \left( H_n( \mathcal{D}_{\Q[G]} \otimes_{\Q[G]} C_{*})\right) = 0 \mbox{ if } \dim_{\Q} \left( H_n(\Q \otimes_{\Q[G]} C_{*})\right) = 0,
\end{equation}
and derived that $b_1^{(2)}(\ker \phi) \leq b_1^{(2)}(\ker f^{*}\phi)$. The main step in our proof is showing that for any (locally indicable) $3$-manifold group we have the following inequality (\cref{thm: 3mfld_Lewin})
\[
	\dim_{\mathcal{D}_{\Q[G]}} \left( \mathcal{D}_{\Q[G]} \otimes_{\Q[G]} L \right) \leq \dim_{\Q} \left( \Q \otimes_{\Q[G]} L \right),
\]
which implies that \labelcref{eqtn: ineq_intro} holds. In other words, \cref{thm: ineq_thurston3} is reduced to a problem on universal division ring embeddings.

\subsection{Lewin groups}

A ring homomorphism $\varphi$ from $R$ to a division ring $\mathcal{D}$ induces a dimension function on finitely presented $R$-modules, namely, $\dim_{\mathcal{D}}(\mathcal{D} \otimes_R L)$ for any finitely presented $R$-module $L$. A division $R$-ring $\varphi \colon R \to \mathcal{D}$ is called \emph{epic} if the division closure of $\im \varphi$ equals to $\mathcal{D}$. It is further called \emph{division $R$-ring of fractions} if $\varphi$ is injective. Cohn introduced the notion of universal division $R$-ring (see, for instance, \cite[Section 7.2]{cohn06FIR}). An epic division $R$-ring $\mathcal{D}$ is \emph{universal} if for every division $R$-ring $\mathcal{E}$, $\dim_{\mathcal{D}}(\mathcal{D} \otimes_R L) \leq \dim_{\mathcal{E}}(\mathcal{E} \otimes_R L)$ for every finitely presented $R$-module $L$. If it exists, the universal epic division $R$-ring of fractions is unique up to $R$-isomorphism \cite[Theorem 7.2.7]{cohn06FIR}. 

Jaikin-Zapirain \cite[Proposition 4.1]{JaikinZapirain2020THEUO} showed that the group algebra $\Q[G]$ does not have a universal division ring of fractions if $G$ is not locally indicable. In this paper, we study whether the group algebra $\Q[G]$ of a $3$-manifold group $G$, or more generally a crossed product $E * G$ where $E$ is a division ring, has a universal division ring of fractions. Thus, it is natural to consider the case where $G$ is a locally indicable $3$-manifold group.

Let $E$ be a division ring, $G$ a locally indicable group and $E * G$ a crossed product. Hughes \cite{HughesDivRings1970} introduced a freeness property for a given epic division $E * G$-ring $\mathcal{D}$ (see \cref{sec: R-rings}), and showed that if $E * G$ admits such embedding, then it is unique up to $E * G$-ring isomorphism. In this situation, we thus speak of \emph{the Hughes-free division ring of $E * G$}, denote it by $\mathcal{D}_{E * G}$, and view $E * G$ as a subring of $\mathcal{D}_{E * G}$. If for every division ring $E$ and every crossed product $E * G$ exists the Hughes-free division ring for $E * G$, we say that $G$ is \emph{Hughes-free embeddable}.

A locally indicable group $G$ is a \emph{Lewin group} if it is Hughes-free embeddable and for all possible crossed products $E * G$, where $E$ is a division ring, $\mathcal{D}_{E * G}$ is universal. In \cite[Conjecture 1]{JaikinZapirain2020THEUO}, Jaikin-Zapirain conjectures that every locally indicable group is a Lewin group. He showed that locally indicable amenable groups, residually-\{torsion-free nilpotent\} groups and free-by-cyclic groups are Lewin groups. In this paper we confirm the conjecture within the realm of $3$-manifold groups.

\begin{thm} \label{thm: 3mfld_Lewin}
	Locally indicable $3$-manifold groups are Lewin groups.
\end{thm}

Prominent examples of rings having universal division ring of fractions are pseudo-Sylvester domains (see \cref{sec: pseudoSyl} for the precise definition). In the case of crossed products $E * G$ these come, for instance, from free and free-by-$\Z$ groups \cite{Cohn_Universality,HennekeLopez_pseudoSylvester}. In \cite{JaikinSouza_Sylvesterprop}, Jaikin-Zapirain and Souza proved that the completed group algebra of a torsion-free finitely generated virtually free-by-$\Z_p$ pro-$p$ group is a Sylvester domain. We offer its $3$-manifold group analog, from which \cref{thm: 3mfld_Lewin} will follow and which we believe is of independent interest. A group $G$ is free-by-cyclic if it fits into a short exact sequence 
\[
	1 \to F \to G \to Z \to 1
\]
where $F$ is free and $Z$ is cyclic. Throughout this article, we do not require the free kernel of free-by-cyclic groups to be finitely generated.

\begin{thm}\label{thm: univ_vFbyZ}
	Let $E$ be a division ring, $G$ a torsion-free $3$-manifold group that is virtually free-by-cyclic, and $E * G$ a crossed product. Then $E * G$ is a pseudo-Sylvester domain. Moreover, the universal division $E*G$-ring of fractions is Hughes-free. In particular, $G$ is a Lewin group.
\end{thm}

We remark that \cref{thm: univ_vFbyZ} already leads to \cref{cor: Simon}, which was the original motivation.

\subsection*{Organization of the paper}

    In \cref{sec: prelims} we provide some background on division ring embeddings, $L^2$-Betti numbers and pseudo-Sylvester domains. In \cref{sec: Lewin} we show \cref{thm: univ_vFbyZ} and deduce \cref{thm: 3mfld_Lewin}. We then use these results to prove \cref{thm: ineq_thurston3} in \cref{sec: Thurston_ineq}. We also show \cref{cor: Thurston_ball} and provide applications to BNS invariants.

\subsection*{Acknowledgments} 
    
    The author is supported by the grants \seqsplit{PID2024-155800NB-C33} and \seqsplit{EUR2025-164928} of the Ministry of Science, Innovation and Universities of Spain. The author is grateful to Andrei Jaikin-Zapirian for enlightening conversations and remarks on this work, and to Alan Reid for interesting questions and discussions that led to \cref{cor: Thurston_ball}. The author also thanks Wolfgang L\"uck and Andrew Ng for useful comments on this work.

\section{Preliminaries} \label{sec: prelims}

Throughout, rings are assumed to be associative and unital, ring homomorphisms preserve the unit and modules are left modules unless otherwise specified.

\subsection{Division ring embeddings} \label{sec: R-rings}

An \textit{$R$-ring} is a pair $(S, \varphi)$ such that $\varphi \colon R \rightarrow S$ is a ring homomorphism. We will often omit $\varphi$ if it is clear from the context. Two $R$-rings $(S_1, \varphi_1)$ and $(S_2, \varphi_2)$ are {\it isomorphic} if there exists a ring isomorphism $\alpha \colon S_1 \to S_2$ such that $\varphi_2 = \alpha \circ \varphi_1$. 

A \emph{division} $R$-ring $(\mathcal{D}, \varphi)$ is an $R$-ring such that $\mathcal{D}$ is a division ring. Note that $\mathcal D$ inherits the structure of an $R$-bimodule, and therefore given an $R$-module $L$ we can compute its homology with coefficients in $\mathcal D$. Recall that every module over $\mathcal D$ has a well-defined dimension. If $R$ is a ring and $G$ is a group, then $R[G]$ denotes the corresponding group ring. Observe that $R$ is naturally an $R[G]$-module with $G$ acting trivially.

\begin{defn}
	Let $E$ be a division ring, $G$ a group and $\mathcal D$ be a division $E[G]$-ring. Given a non-negative integer $n$, the $n$th \emph{$\mathcal D$-Betti number} of $G$ is the (possibly infinite) value
	\[
		b_n^\mathcal D (G) := \dim_\mathcal D \left( \Tor_n^{E[G]} (\mathcal D, E) \right) \in \Z_{\geq 0} \cup \{\infty\}.
	\]
\end{defn}

An \emph{epic} division $R$-ring is a division $R$-ring $(\mathcal{D}, \varphi)$ such that the division closure of $\im \varphi$ equals $\mathcal{D}$. It is further called \emph{division $R$-ring of fractions} if $\varphi$ is injective.

Let $G$ be a group. A ring $S$ is $G$-\emph{graded} if $S = \bigoplus_{g\in G} S_g$ as an additive group, where $S_g$ is an additive subgroup for every $g\in G$, and $S_g S_h \subseteq S_{gh}$ for all $g,h \in G$. If $S_g$ contains an invertible element $u_g$ for each $g\in G$, then we say that $S$ is a \emph{crossed product} of $S_e$ and $G$ and we shall denote it by $S = S_e * G$.

A group $G$ is \emph{locally indicable} if all its non-trivial finitely generated subgroups admit an epimorphism onto $\Z$. Let $E$ be a division ring, $G$ a locally indicable group and $E * G$ a crossed product. An epic division $E * G$-ring $(\mathcal{D}, \varphi)$ is called \emph{Hughes-free} if for every non-trivial finitely generated subgroup $H \leqslant G$ and every $U \triangleleft H$ such that $H/U \cong \Z$, the multiplication map 
\[
\mathcal D_U \otimes_{E * U} E * H \rightarrow \mathcal D, \qquad x \otimes y \mapsto x \cdot \varphi(y)
\]
is injective. Here, $\mathcal D_U$ denotes the division closure of $\varphi(E * U)$ in $\mathcal D$. By uniquiness up to ring isomorphism \cite{HughesDivRings1970}, if $H$ is any subgroup of $G$, then $\mathcal{D}_{E * H}$ is isomorphic to the division closure of $E * H$ in $\mathcal{D}_{E * G}$, so we view $\mathcal{D}_{E * H}$ as a subring of $\mathcal{D}_{E * G}$. These division rings arise naturally in the context of $L^2$-theory.

We recall that the \emph{augmentation ideal} of $R[G]$ is the two-sided ideal $I_{R[G]}$ defined as the kernel of the \emph{agumentation map} $\varepsilon_{R}$ from $R[G]$ to $R$ that sends $\sum_{g \in G} a_g g$ to $\sum_{g \in G} a_g$.

\subsection{\texorpdfstring{$L^2$}{L²}-Betti numbers} \label{sec: L2Betti}

In this section we give a brief review of $L^2$-Betti numbers, emphasising what will be important for us. We refer the reader to \cite{Luck02} for more background on $L^2$-invariants. Let $G$ be a countable group and let $\ell^2(G)$ denote the Hilbert space of complex series in elements of $G$ with square-summable coefficients. The left and right multiplication actions of $G$ on itself extend to left and right actions of $G$ on $\ell^2(G)$. We can further extend the right action of $G$ on $\ell^2(G)$ to an action of $\C[G]$ on $\ell^2(G)$, and hence consider $\C[G]$ as a subalgebra of $\mathcal{B}(\ell^2(G))$, namely the {\it bounded linear operators on} $\ell^2(G)$. The {\it group von Neumann algebra} $\mathcal N(G)$ of $G$ is the algebra of bounded operators on $\ell^2(G)$ that commute with the left action of $G$. Note that $\mathcal N(G)$ contains the complex group algebra $\C[G]$. The ring $\mathcal{N}(G)$ is a finite von Neumann algebra. Moreover, $\mathcal{N}(G)$ satisfies the left (and right) Ore condition with respect to the set of non-zero-divisors (a result proved by S. K. Berberian in \cite{Berberian82}). We denote by $\mathcal{U}(G)$ the left (resp. right) classical ring of fractions named the {\it ring of unbounded operators affiliated to} $G$. Using the trace functional of $\mathcal{N}(G)$, one can define a dimension function on all $\mathcal{N}(G)$-modules (see \cite[Chapters 6.1 and 6.2]{Luck02}), and define for every non-negative integer $n$ the $n$th \emph{$L^2$-Betti number} of $G$ to be
\[
	b_n^{(2)}(G) := \dim_{\mathcal N(G)} H_n(G; \mathcal N(G)).
\]

The description of $L^2$-Betti numbers given so far relies on the involved definition of $\dim_{\mathcal N(G)}$. However, we will work with groups that satisfy the Strong Atiyah Conjecture, in which case the definition is far more straightforward. We give a formulation of the Strong Atiyah Conjecture for torsion-free groups. Given a subfield $k$ of $\C$, the \emph{Linnell ring over $k$}, denoted $\mathcal{D}(k[G])$, is the division closure of $k[G]$ in $\mathcal{U}(G)$.

\begin{conj}[The Strong Atiyah Conjecture over $k \subseteq \C$]
	Let $G$ be a torsion-free countable group, and let $k$ be a subfield of $\C$. The Linnell ring $\mathcal{D}(k[G])$ is a division ring.
\end{conj}

Linnell and Schick showed that this statement is equivalent to the usual formulation of the Strong Atiyah Conjecture in terms of the von Neumann dimension function (see \cite[Lemma 12.3]{LinnellZeroDivc} and \cite[Lemma 3]{SchickIntegrality}). While the Strong Atiyah Conjecture is open in general, it has been established for many large classes of groups (including torsion-free $3$-manifold groups \cite{FriedlLuck_euler,KielakLinton_3mfldAtiyah}). For ease of notation, we write $\mathcal{D}_G$ for the Linnell ring over $\Q$ of a group $G$ satisfying the Strong Atiyah Conjecture. In this case, we have
\[
	b_n^{(2)}(G) = \dim_{\mathcal{D}_G} H_n(G; \mathcal{D}_G) = b_n^{\mathcal{D}_G}(G)
\]
where $\dim_{\mathcal{D}_G}$ is just the usual rank of a module over a division ring.

When $G$ is locally indicable, $\mathcal D_{\Q[G]}$ always exists and coincides with $\mathcal{D}_G$ by a result of Jaikin-Zapirain and L\'opez-\'Alvarez \cite{JaikinLopez_Atiyah}. Note that by uniqueness up to $\Q[G]$-isomorphism, if $G$ is a Lewin group then $\mathcal{D}_G$ is in fact the universal (Hughes-free) division $\Q[G]$-ring of fractions.

For a general countable group $G$, an epic division $E * G$-ring $(\mathcal{D}, \varphi)$ is called \emph{Linnell-free} if for every pair of subgroups $U \leqslant H$ of $G$, the multiplication map 
\[
\mathcal D_U \otimes_{E * U} E * H \rightarrow \mathcal D, \qquad x \otimes y \mapsto x \cdot \varphi(y)
\]
is injective. It is shown in \cite{Linnell_localizations} (cf. the discussion after Problem 4.5) that elements in different cosets of $H$ in $G$ are right $\UO(H)$-linearly independent in $\UO(G)$. Thus, if $G$ satisfies the Strong Atiyah Conjecture, $\mathcal{D}_G$ is Linnell-free. However, while \cite{HughesDivRings1970} ensures the uniqueness of the division ring for locally indicable groups, it remains unknown whether Linnell-free division $E*G$-rings are unique up to $E*G$-isomorphism for a general group $G$. Moreover, Gr\"ater showed in \cite[Corollary 8.3]{Grater20} that if $G$ is a locally indicable group and the Hughes-free division $E* G$-ring $\mathcal D_{E*G}$ exists, then it is in fact a Linnell-free division $E*G$-ring.

\subsection{Pseudo-Sylvester domains} \label{sec: pseudoSyl}

A ring $R$ is said to be \emph{stably finite} if whenever $A$ and $B$ are two $n \times n$-matrices over $R$ such that $AB = I_n$, then also $BA = I_n$. We remark that if $E$ is a field of characteristic zero or $E$ is of positive characteristic and $G$ is sofic, then $E[G]$ is stably finite \cite{Kaplansky72,ElekSzabo04}.

Let $R$ be a ring and let $A$ be an $m \times n$ matrix over $R$. The \emph{inner rank} of $A$, denoted $\rho(A)$, is defined as the least $k$ such that $A$ admits a decomposition $A = B_{m \times k} C_{k \times n}$. Recall also that the \emph{stable rank} of $A$, denoted $\rho^{*}(A)$, is given by
\[
\rho^{*}(A) : = \lim_{s \to \infty} (\rho(A \oplus I_s) - s),
\]
whenever the limit exists, where $A \oplus I_s$ denotes the block diagonal matrix with blocks $A$ and $I_s$. When $R$ is stably finite, $\rho^{*}(A)$ is well-defined and non-negative, and it is positive if $A$ is a non-zero matrix \cite[Proposition 0.1.3]{cohn06FIR}.

A ring $R$ is a \emph{pseudo-Sylvester domain} if $R$ is stably finite and satisfies the law of nullity with respect to the stable rank, that is, if $A \in \Mat_{m \times n}(R)$ and $B \in \Mat_{n \times k}(R)$ are such that $A B = 0$, then
\[
\rho^{*}(A) + \rho^{*}(B) \leq n.
\]

\begin{thm}[{\cite[Theorem 7.5.18]{cohn06FIR}}] \label{thm: univ_pseudoSyl}
	Every pseudo-Sylvester domain $R$ has a universal division $R$-ring of fractions.
\end{thm}

\section{Locally indicable \texorpdfstring{$3$}{3}-manifold groups are Lewin} \label{sec: Lewin}

In this section we show \cref{thm: 3mfld_Lewin}. As a consequence, we obtain that the fundamental group of an admissible $3$-manifold with positive first Betti number is a Lewin group. \cref{thm: 3mfld_Lewin} will follow from \cref{thm: univ_vFbyZ} on torsion-free $3$-manifold groups that are virtually free-by-cyclic. 

\subsection{Torsion-free virtually free-by-cyclic \texorpdfstring{$3$}{3}-manifold groups}

In this section $G$ stands for a torsion-free $3$-manifold group that is virtually free-by-cyclic, $E$ for a division ring and $E * G$ for a crossed product. Note that if $G$ is virtually free, then by \cite[Theorem B]{Swancd1} $G$ is actually a free group. In this case, $E * G$ is known to be a free ideal ring \cite{Cohn_Universality}. Consequently, the results listed below--including \cref{thm: univ_vFbyZ,thm: 3mfld_Lewin}--are already established and admit more straightforward proofs. For the remainder of the section, let $H$ denote a finite-index normal subgroup of $G$ that is free-by-cyclic (with the free-by-$\Z$ case being of primary interest). Note that $H$ is itself a $3$-manifold group. 

To show that $E * G$ is a pseudo-Sylvester domain (\cref{thm: univ_vFbyZ}), we shall use the following criterion of Henneke and L\'opez-\'Alvarez. We recall that an $R$-module $L$ is called \emph{stably free} if there exists a non-negative integer $n$ such that $L \oplus R^n$ is a free $R$-module.

\begin{thm}[{\cite[Theorem 3.4]{HennekeLopez_pseudoSylvester}}] \label{thm: pseudoSyl_criterion}
	Let $\mathcal{D}$ be a division $R$-ring of fractions. Assume that
	\begin{enumerate}[label=(\roman*)]
		\item $\Tor_1^R(\mathcal{D}, \mathcal{D}) = \{0 \}$ and
		\item for any finitely generated left or right $R$-submodule $L$ of $\mathcal{D}$ and any exact sequence $0 \to J \to R^n \to L \to 0$, the left or right $R$-module $J$ is finitely generated stably free.
	\end{enumerate}
	Then $R$ is a pseudo-Sylvester domain and $\mathcal{D}$ is the universal division $R$-ring of fractions.
\end{thm}

In order to apply \cref{thm: pseudoSyl_criterion}, we will first construct a Hughes-free division ring embedding for $E * G$, and then check that this embedding satisfies the conditions of the theorem. 

Before going into the proofs, we start showing that $G$ is indeed a locally indicable group. The following lemma is well-known in the aspherical orientable case; we include a proof because we have not seen it appear without these assumptions.

\begin{prop} \label{prop: cd2_li}
	If $M$ is a connected $3$-manifold with $\pi_1(M)$ torsion-free and $\cd_{\Q}(\pi_1(M)) \leq 2$, then $\pi_1(M)$ is locally indicable.
\end{prop}

\begin{proof}
	Without loss of generality we assume that $\pi_1(M)$ is finitely generated. By \cite[Theorem 1.1]{KielakLinton_3mfldAtiyah}, $\pi_1(M)$ is virtually free-by-cyclic, and hence, $b_2^{(2)}(\pi_1(M))$ vanishes. By the work of Friedl--L\"uck and Kielak--Linton \cite{FriedlLuck_euler, KielakLinton_3mfldAtiyah} $3$-manifold groups lie in Linnell's class $\mathcal{C}$, and hence, being torsion-free, $\pi_1(M)$ satisfies the Strong Atiyah Conjecture over $\C$ by \cite{LinnellDivRings93}. Therefore, $\pi_1(M)$ is locally indicable by \cite[Theorem 3.11]{JaikinLinton_oneRelCoherence}.
\end{proof}

\begin{cor} \label{cor: vFbyZ_li}
	Let $G$ be a torsion-free $3$-manifold group that is virtually free-by-cyclic. Then $G$ is a locally indicable group of cohomological dimension at most $2$.
\end{cor}
\begin{proof}
	Since $G$ is torsion-free and $H$ is of finite index in $G$, according to \cite[Th\'eor\`eme 1, Section 1.7]{Serre_CohomologyFI} $\cd(G) = \cd(H) \leq 2$. Therefore, by \cref{prop: cd2_li} we get that $G$ is locally indicable.
\end{proof}

Now, we construct the Hughes-free embedding. We will use a criterion due to Fisher and the author \cite[Theorem 4.2]{FisherPeralta_Kaplansky'sZD3mfld}.

\begin{thm} \label{thm: HF_vFbyZ}
	Let $G$ be a torsion-free $3$-manifold group that is virtually free-by-cyclic. Then $G$ is Hughes-free embeddable. Moreover, if $H$ denotes a finite-index normal subgroup of $G$ such that $H$ is free-by-cylic, then for any crossed product $E * G$, where $E$ is a division ring, the division $E*G$-ring $\mathcal{D}_{E * G}$ is of the form $\mathcal{D}_{E * H} * G/H$.
\end{thm}
\begin{proof}
	Fix a division ring $E$. By \cite[Theorem 3.7(1)]{JaikinZapirain2020THEUO} and \cite[Corollary 8.3]{Grater20}, we assume without loss of generality that $G$ is finitely generated. According to \cite[Theorem 1.1]{JaikinZapirain2020THEUO}, $H$ is a Lewin group, and so, there exists the Hughes-free division $E * H$-ring $\mathcal{D}_{E * H}$. 
	
	Now, by \cite[Section 3.5]{Cav_Good} the fundamental group of a compact $3$-manifold group is good (in the sense of Serre). Hence, it follows from Scott Core Theorem \cite{Scott_core} that $H$ is good. Moreover, $H$ has the factorization property according to \cite[Lemma 2.4]{Schreve_AtiyahVCS}. Thus, by \cite[Theorem 4.2]{FisherPeralta_Kaplansky'sZD3mfld} (we only need finite generation and finite cohomological dimension in \cite[Theorem 3.7]{FriedlSchreveTillmann_ThurstonFox}), there exists the Hughes-free division $E * G$-ring $\mathcal{D}_{E * G}$, and moreover, it is of the form
	\[
		\mathcal{D}_{E * G} = \mathcal{D}_{E * H} * G/H.
	\]
\end{proof}

Next, we examine the properties of the Hughes-free embedding. Recall that over a ring $R$, an $R$-module $L$ is \emph{coherent} if every finitely generated $R$-submodule $J$ of $L$ is in fact finitely presented as an $R$-module.

\begin{prop} \label{prop: coherence_D}
	Let $E$ be a division ring, $G$ a torsion-free $3$-manifold group that is virtually free-by-cyclic, and $E * G$ a crossed product. Then the Hughes-free division $E * G$-ring $\mathcal{D}_{E * G}$ is coherent.
\end{prop}
\begin{proof}
	Let $L$ be a finitely generated $E * G$-submodule of $\mathcal{D}_{E * G}$, which is of the form $\mathcal{D}_{E * H} * G/H$ by \cref{thm: HF_vFbyZ}. Note that $L$ is also a finitely generated $E * H$-submodule of $\mathcal{D}_{E * H}^d$, where $d = \lvert G : H \rvert$. By \cite[Proposition 3.9]{HennekeLopez_pseudoSylvester} $\mathcal{D}_{E * H}$, and hence $\mathcal{D}_{E * H}^d$, is coherent as $E * H$-module. Therefore, $L$ is finitely presented as an $E * H$-module, and consequently as an $E * G$-module.
\end{proof}

Recall that an $R$-module $L$ is of \emph{weak dimension at most $1$} if for every right $R$-module $J$ the module $\Tor_2^R(J, L)$ vanishes. Similarly, an $R$-module $L$ is of \emph{projective dimension at most $1$} if for every $R$-module $J$ the module $\Ext_R^2(L, J)$ vanishes. There is a corresponding right-module version of the following result:

\begin{prop} \label{prop: homology_D}
	Let $E$ be a division ring, $G$ a torsion-free $3$-manifold group that is virtually free-by-cyclic, and $E * G$ a crossed product. Then the Hughes-free division $E * G$-ring $\mathcal{D}_{E * G}$ satisfies the following properties:
	\begin{enumerate}[label=(\roman*)]
		\item $\Tor_1^{E * G}(\mathcal{D}_{E * G}, \mathcal{D}_{E * G}) = \{0\}$.
		\item The $E*G$-module $\mathcal{D}_{E * G}$ is of weak dimension at most $1$.
		\item Every finitely generated $E * G$-submodule of $\mathcal{D}_{E * G}$ is of projective dimension at most $1$.
	\end{enumerate}
\end{prop}
\begin{proof}
	Let $L$ be a right $E * G$-module. By \cref{thm: HF_vFbyZ}, using the Shapiro lemma, we get that
	\[
	\Tor_n^{E * G}(L, \mathcal{D}_{E * G}) \cong \Tor_n^{E *H}(L, \mathcal{D}_{E *H})
	\]
	for every non-negative integer $n$. To show \emph{(i)}, we use that the module $\Tor_1^{E * H}(\mathcal{D}_{E * H}, \mathcal{D}_{E * H})$ vanishes by \cite[Lemma 3.8(7)]{HennekeLopez_pseudoSylvester}. Hence,
	\[
	\Tor_1^{E * G}(\mathcal{D}_{E * G}, \mathcal{D}_{E * G}) \cong \bigoplus_{\lvert G: H \rvert} \Tor_1^{E * H}(\mathcal{D}_{E * H}, \mathcal{D}_{E * H}) = \{0\}.
	\]
	For the proof of \emph{(ii)}, note that every module is a directed union of its finitely generated submodules. Thus, it follows from \cite[Lemma 3.7(4)]{HennekeLopez_pseudoSylvester} that $\mathcal{D}_{E * H}$ is of weak dimension at most $1$ as an $E * H$-module, and so
	\[
	\Tor_2^{E * G}(L, \mathcal{D}_{E * G}) \cong \Tor_2^{E *H}(L, \mathcal{D}_{E *H}) = \{0\}.
	\]
	To show \emph{(iii)} assume that $L$ is a finitely generated $E * G$-submodule of $\mathcal{D}_{E * G}$. By \cref{prop: coherence_D}, $L$ is finitely presented. Moreover, since $E * G$ has global dimension at most $2$ and $\mathcal{D}_{E * G}$ is of weak dimension at most $1$, $L$ is also of weak dimension at most $1$. Thus, for any exact sequence of $E * G$-modules
	\[
	0 \to J \to (E* G)^n \to L \to 0,
	\]
	$J$ is finitely generated and flat. Another application of \cref{prop: coherence_D} shows that $J$ is also finitely presented, and hence, $J$ is a projective $E * G$-module by \cite[Theorem 3.63]{Rotman09} as we wanted.
\end{proof}

We will further need the following result on the structure of finitely generated projective $E * G$-modules.

\begin{thm}\label{thm: FJ_vFbyZ}
	Let $E$ be a division ring, $G$ a $3$-manifold group and $E * G$ a crossed product. Then every finitely generated projective $E * G$-module is stably free.
\end{thm}
\begin{proof}
	By combined work of Roushon and Bartels--Farrell--L\"uck  \cite{Roushon_FJ3mfld,BartelsFarrellLuck_FJlattices} $G$ satisfies the $K$-theoretic Farrell--Jones Conjecture with coefficients in an additive category. Hence, the statement follows from \cite[Proposition 4.2]{HennekeLopez_pseudoSylvester}.
\end{proof}

We conclude the section showing that $E * G$ is a pseudo-Sylvester domain and that $G$ is a Lewin group.

\begin{proof}[Proof of \cref{thm: univ_vFbyZ}]
	Consider the Hughes-free division $E * G$-ring $\mathcal{D}_{E * G}$ given by \cref{thm: HF_vFbyZ}. We will check that it satisfies the conditions of \cref{thm: pseudoSyl_criterion}. By \cref{prop: coherence_D,prop: homology_D} the module $\Tor_1^{E * G}(\mathcal{D}_{E *G}, \mathcal{D}_{E * G})$ is trivial and every finitely generated $E * G$-submodule $L$ of $\mathcal{D}_{E * G}$ is a finitely presented module of projective dimension at most $1$. Thus, for any exact sequence of $E * G$-modules
	\[
	0 \to J \to (E * G)^n \to L \to 0,
	\]
	we conclude that $J$ is a finitely generated projective module. Hence, by \cref{thm: FJ_vFbyZ}, $J$ is stably free. Therefore, \cref{thm: pseudoSyl_criterion} implies that $E * G$ is a pseudo-Sylvester domain and that the Hughes-free division $E * G$-ring $\mathcal{D}_{E * G}$  is the universal division $E * G$-ring of fractions. In particular, $G$ is a Lewin group.
\end{proof}

\subsection{Locally indicable \texorpdfstring{$3$}{3}-manifold groups}

Building on the work of the previous section, we now show that locally indicable $3$-manifold groups are Lewin. Importantly, this class of groups includes every fundamental group $\pi_1(M)$ of an admissible $3$-manifold $M$ with positive first Betti number.

\begin{prop}\label{prop: loc_ind_3mfld}
	Let $M$ be a compact, connected $3$-manifold with $\pi_1(M)$ locally indicable. Then $\pi_1(M)$ is an extension of a locally \{virtually free-by-cyclic\} group and $\Z$.
\end{prop}
\begin{proof}
	Since any $3$-manifold that is prime but not irreducible is homeomorphic to $S^2 \times S^1$ or $S^2 \widetilde{\times} S^1$, which have fundamental group $\Z$, by the Prime Decomposition Theorem we obtain a finitely generated free group $F$ and finitely many compact, connected and irreducible $3$-manifolds $M_1, \ldots, M_k$ such that $G \cong F \ast (\ast_{i=1}^k \pi_1(M_i))$. Moreover, since $G$ is torsion-free, each $M_i$ is aspherical by the Sphere Theorem. 
	
	Now, each $\pi_1(M_i)$ is locally indicable and finitely generated, so there exists a map $\varphi_i$ from $\pi_1(M_i)$ onto $\Z$. These maps induce a map $\varphi$ from $G$ onto $\Z$ that sends the elements of $F$ to $0$ and the elements of $\pi_1(M_i)$ via $\varphi_i$. We claim that $H := \ker \varphi$ is locally \{virtually free-by-cyclic\}. Indeed, denote by $\widehat{M_i}$ the infinite cyclic cover associated to $\varphi_i$. Since $M_i$ is compact and aspherical, $\cd(\pi_1(\widehat{M_i})) \leq 2$. Thus, it follows from Bass-Serre theory that $H$ splits as a free product of groups of cohomological dimension at most $2$. In particular, $\cd(H) \leq 2$, and hence, by \cite[Theorem 1.1]{KielakLinton_3mfldAtiyah}, $H$ is locally \{virtually free-by-cyclic\}.
\end{proof}

\begin{proof}[Proof of \cref{thm: 3mfld_Lewin}]
	Let $M$ be a connected $3$-manifold and assume that $G := \pi_1(M)$ is locally indicable. By \cite[Theorem 3.7(1)]{JaikinZapirain2020THEUO} it suffices to consider the case where $G$ is finitely generated. So, by Scott Core Theorem \cite{Scott_core}, we may assume that $M$ is compact. Now, by \cref{prop: loc_ind_3mfld} we have that $G$ is a group extension of a locally \{virtually free-by-cyclic\} group $H$ and $\Z$. According to \cite[Theorem 3.7(3)]{JaikinZapirain2020THEUO}, it is enough to show that $H$ is a Lewin group. But another application of \cite[Theorem 3.7(1)]{JaikinZapirain2020THEUO} reduces the proof to showing that $3$-manifold groups that are torsion-free and virtually free-by-cyclic are Lewin groups. Therefore, the result follows from \cref{thm: univ_vFbyZ}. 
\end{proof}

\begin{cor} \label{cor: admissible_Lewin}
	Let $M$ be an admissible $3$-manifold such that $b_1(\pi_1(M))$ is positive. Then $\pi_1(M)$ is a Lewin group.
\end{cor}
\begin{proof}
	By work of Howie and Short \cite[Theorem 3.1]{BRW_orderable3gps} $\pi_1(M)$ is locally indicable. Therefore, $\pi_1(M)$ is a Lewin group by \cref{thm: 3mfld_Lewin}.
\end{proof}

\section{Inequality of the Thurston norm} \label{sec: Thurston_ineq}

In this section we show \cref{thm: ineq_thurston3}. Our strategy follows the method of Friedl and L\"uck in \cite{FriedlLuck_euler} where they introduce the notion of $(\mu, \phi)$-$L^2$-Euler characteristic. Let us explain the connection between $\mathcal{D}$-Betti numbers and the $(\mu, \phi)$-$L^2$-Euler characteristic. We begin by recalling its definition. Let $X$ be a connected CW-complex and let $H$ denote its fundamental group. Let $\mu \colon H \to G$ and $\phi \colon G \to \Z$ be group homomorphism, and let $\overline{X} \to X$ be the $G$-covering associated to $\mu$. For every $n \geq 0$, define
\[
	b_n^{(2)}(\overline{X}; \VN(G); \phi) := \dim_{\VN(G)} \left( H_n \left( \VN(G) \otimes_{\Z[G]} \left( C_{*}(\overline{X}) \otimes_{\Z} \Z[\Z] \right) \right) \right)
\]
where $C_{*}(\overline{X})$ is the cellular $\Z[G]$-chain complex of $\overline{X}$, $G$ acts on $\Z[\Z]$ via $\phi$ and $G$ acts on $ C_{*}(\overline{X}) \otimes_{\Z} \Z[\Z]$ diagonally. If the sum $\sum_{n \geq 0} b_n^{(2)}(\overline{X}; \VN(G); \phi)$ is finite, then the \emph{$(\mu, \phi)$-$L^2$-Euler characteristic} is defined as the value
\[
	\chi^{(2)}(X; \mu, \phi) := \sum_{n \geq 0} (-1)^n  b_n^{(2)}(\overline{X}; \VN(G); \phi).
\]
Now, if $\phi$ is non-trivial, then the image of $\phi$ has finite index in $\Z$. Let $m_{\phi}$ be the index and let $U_1$ denote the kernel of $\phi$. As $\Z[G]$-module, $\Z[\Z]$ is isomorphic to $\Z[G/U_1]^{m_{\phi}}$. Thus, if $i^{*} \overline{X}$ denotes the $U_1$-CW-complex obtained from $\overline{X}$ by restriction of the action of $G$, for every $n \geq 0$
\[
	b_n^{(2)}(\overline{X}; \VN(G); \phi) = m_{\phi} \dim_{\VN(U_1)} \left( H_n \left( \VN(U_1) \otimes_{\Z[U_1]} C_{*}(i^{*}\overline{X}) \right) \right) 
\]
(see, for instance, \cite[Lemma 2.6]{FriedlLuck_euler} for more details). Set $\phi_2 := \phi \circ \mu$ and denote by $U_2$ the kernel of $\phi_2$. As right $\Z[H]$-modules
\[
	\VN(U_1) \otimes_{\Z[U_2]} \Z[H] \cong \VN(U_1) \otimes_{\Z[U_1]} \Z[G]
\]
where $U_2$ acts on $\VN(U_1)$ via $\mu$. Let $\widetilde{X} \to X$ be the $H$-covering. Then, since $C_{*}(\overline{X})$ is given by the complex $\Z[G] \otimes_{\Z[H]} C_{*}(\widetilde{X})$, we get that
\[
	\dim_{\VN(U_1)} \left( H_n \left( \VN(U_1) \otimes_{\Z[U_1]}  C_{*}(i^{*}\overline{X}) \right) \right)
\]
equals to
\[
	\dim_{\VN(U_1)} \left( H_n \left( \VN(U_1) \otimes_{\Z[U_2]}  C_{*}(i^{*} \widetilde{X}) \right) \right).
\]
Thus, putting all together we get that
\[
	\chi^{(2)}(X; \mu, \phi) = m_{\phi} \sum_{n \geq 0} (-1)^n \dim_{\VN(U_1)} \left( \Tor_n^{\Z[U_2]}(\VN(U_1), \Z)\right).
\]
Note that if $U_1$ satisfies the Strong Atiyah Conjecture, then it holds that the $(\mu, \phi)$-$L^2$-Euler characteristic is equal to the following alternate sum of $\mathcal{D}$-Betti numbers
\begin{equation}\label{eqtn: twisted_L2_Euler}
	\chi^{(2)}(X; \mu, \phi) = m_{\phi} \sum_{n \geq 0} (-1)^n b_n^{\mathcal{D}_{U_1}}(U_2).
\end{equation}
In any case, if $\mu = \id_H$ and $\phi$ is surjective, then $U_1 = U_2$ and $m_{\phi} = 1$, so under these circumstances $\chi^{(2)}(X; \mu, \phi)$ is simply $\chi^{(2)}(U_1)$ (recall that $U_1 = \ker \phi$). If in addition, we take $X$ to be an aspherical $3$-manifold, then $U_1$ is associated to an infinite cyclic cover, and hence $\cd(U_1) \leq 2$. Therefore, as a consequence of \cite[Theorem 1.1]{KielakLinton_3mfldAtiyah}, we get that $b_n^{(2)}(U_1) = 0$ for all $n \neq 1$. Thus, \cite[Theorem 0.2]{FriedlLuck_euler} states that for an admissible $3$-manifold $M$ and a surjective $\phi \in H^{1}(M;\Z)$
\begin{equation} \label{eqtn: Thurston_beta1}
	x_M(\phi) = \chi^{(2)}(M; \id, \phi) = b_{1}^{(2)}(\ker \phi).
\end{equation}
This equality reduces \cite[Conjecture 0.5]{FriedlLuck_euler} to an inequality of first $\mathcal{D}$-Betti numbers. As we mentioned in the introduction, our approach reframes this inequality as a problem concerning the existence of a universal epic division $\Q[G]$-ring. Since we have proved in \cref{cor: admissible_Lewin} that every fundamental group $G$ of an admissible $3$-manifold with positive $b_1(G)$ is a Lewin group, we are able to eliminate the residually-\{locally indicable and elementary amenable\} condition on the group structure and establish \cref{thm: ineq_thurston3} in full generality.

\begin{thm} \label{thm: L2Euler_inequality}
	Let $f \colon M \to N$ be a map between admissible $3$-manifolds which induces an epimorphism on the fundamental groups and an isomorphism $f_{*} \colon H_n(M; \Q) \to H_n(N; \Q)$ for every non-negative integer $n$. Denote by $H$ and $G$ the fundamental group of $M$ and $N$, respectively, and let $\phi \colon G \to \Z$ be a non-trivial group homomorphism. Then 
	\[
		b_1^{\mathcal{D}_G}(H) = 0.
	\]
\end{thm}
\begin{proof}
	First, observe that since $\phi$ is non-trivial, then $b_1(G)$ is non-zero. Thus, by \cref{cor: admissible_Lewin} $G$ is a Lewin group. Now, consider the short exact sequence given by the augmentation map
	\[
		0 \to I_{\Q[H]} \to \Q[H] \to \Q \to 0.
	\]
	Note that $\pi_1(f)$ induces a natural map $\pi \colon \Q[H] \to \Q[G]$. Changing scalars to $\Q[G]$ yields the short exact sequence
	\begin{equation} \label{eqtn: augmentation}
		0 \to \Tor_1^{\Q[H]}(\Q[G], \Q) \to \Q[G] \otimes_{\Q[H]} I_{\Q[H]} \to I_{\Q[G]} \to 0
	\end{equation}
	where $g \otimes (h -1)$ is mapped to $g(\pi_1(f)(h) - 1)$.

	\begin{claim} \label{claim: fact}
		The map on the second rational homology factors through the module $\Tor_1^{\Q[G]}(\Q, \Q[G] \otimes_{\Q[H]} I_{\Q[H]})$, that is, the following diagram commutes
		\[
			\begin{tikzcd}
				 \Tor_2^{\Q[H]}(\Q, \Q) \arrow[rr] \arrow[dr] &  & \Tor_2^{\Q[G]}(\Q, \Q) \\ & \Tor_1^{\Q[G]}(\Q, \Q[G] \otimes_{\Q[H]} I_{\Q[H]}) \arrow[ur] &
			\end{tikzcd}.
		\]
	\end{claim}
	\begin{proof}
		Let $\pi_{\alpha} \colon \Q[H]^{\alpha} \to \Q[G]^{\alpha}$ be the map defined by applying $\pi$ componentwise. Consider a presentation of $I_{\Q[H]}$ given by the exact sequence of $\Q[H]$-modules
		\[
			0 \to L \xrightarrow{\iota} \Q[H]^{\alpha} \to I_{\Q[H]} \to 0.
		\]
		Note that $\pi_{\alpha}$ induces the exact sequence of $\Q[G]$-modules 
		\[
			0 \to J \to \Q[G]^{\alpha} \to I_{\Q[G]} \to 0
		\]
		where $\pi_{\alpha}(L) \subseteq J$. Thus, the module $\Tor_2^{\Q[H]}(\Q, \Q)$ is the kernel of the map given by change of scalars to $\Q$ on the presentation of $I_{\Q[H]}$
		\[
			\Q \otimes_{\Q[H]} L \to \Q \otimes_{\Q[H]} \Q[H]^{\alpha}.
		\]
		Similarly, $\Tor_2^{\Q[G]}(\Q, \Q)$ is the kernel of the map given by change of scalars to $\Q$ on the presentation of $I_{\Q[G]}$
		\[
			\Q \otimes_{\Q[G]} J \to \Q \otimes_{\Q[G]} \Q[G]^{\alpha}.
		\]
		Hence, the natural map from $\Tor_2^{\Q[H]}(\Q, \Q)$ to $\Tor_2^{\Q[G]}(\Q, \Q)$ is given by restriction of the map 
		\[
			\Q \otimes_{\Q[H]} L \to \Q \otimes_{\Q[G]} J
		\]
		that sends $1 \otimes l$ to $1 \otimes \pi_{\alpha}(l)$. But this map clearly factors through the module $\Q \otimes_{\Q[G]} \pi_{\alpha}(L)$, and so, we only need to show that the kernel of the map 
		\[
			\Q \otimes_{\Q[G]} \pi_{\alpha}(L) \to \Q \otimes_{\Q[G]} \Q[G]^{\alpha}
		\]
		is $\Tor_1^{\Q[G]}(\Q, \Q[G] \otimes_{\Q[H]} I_{\Q[H]})$.
		
		Now, from a change of scalars to $\Q[G]$ on the presentation of $I_{\Q[H]}$ we get the exact sequence of $\Q[G]$-modules
		\[
			0 \to \Q[G]L \to \Q[G] \otimes_{\Q[H]} \Q[H]^{\alpha} \to\Q[G] \otimes_{\Q[H]} I_{\Q[H]} \to 0 
		\]
		where $\Q[G] L$ stands for the image of the map $1 \otimes \iota$ from $\Q[G] \otimes_{\Q[H]} L$ to $\Q[G] \otimes_{\Q[H]} \Q[H]^{\alpha}$. Note that $\Q[G] \otimes_{\Q[H]} \Q[H]^{\alpha}$ is isomorphic to $\Q[G]^{\alpha}$ via the map $\pi_{\alpha}$, and moreover, under this isomorphism $\Q[G]L$ becomes identified with $\pi_{\alpha}(L)$. Thus, we have the following presentation of $\Q[G] \otimes_{\Q[H]} I_{\Q[H]}$
		\[
			0 \to \pi_{\alpha}(L) \to \Q[G]^{\alpha} \to\Q[G] \otimes_{\Q[H]} I_{\Q[H]} \to 0
		\]
		which shows that $\Tor_1^{\Q[G]}(\Q, \Q[G] \otimes_{\Q[H]} I_{\Q[H]})$ is the kernel of the map
		\[
			\Q \otimes_{\Q[G]} \pi_{\alpha}(L) \to \Q \otimes_{\Q[G]} \Q[G]^{\alpha}.
		\]
		 \renewcommand \qedsymbol{$\diamond$} \qedhere
	\end{proof}
	Next, a change of scalars to $\Q$ on sequence \labelcref{eqtn: augmentation} yields the exact sequence
	\begin{equation}\label{eqtn: long_exact_seq}
		\begin{aligned}
			&\Tor_1^{\Q[G]}(\Q, \Q[G] \otimes_{\Q[H]} I_{\Q[H]}) \to \Tor_2^{\Q[G]}(\Q, \Q) \to \\
			&\hspace{2cm} \Q \otimes_{\Q[G]} \Tor_1^{\Q[H]}(\Q[G], \Q) \to \Q \otimes_{\Q[H]} I_{\Q[H]} \to \\
			&\hspace{8cm} \Q \otimes_{\Q[G]} I_{\Q[G]} \to 0.
		\end{aligned}
	\end{equation}
	By assumption $f_{*} \colon H_n(M; \Q) \to H_n(N; \Q)$ is an isomorphism for every $n \geq 0$. In particular, the natural maps from $\Tor_i^{\Q[H]}(\Q, \Q)$ to $\Tor_i^{\Q[G]}(\Q, \Q)$ for $i = 1,2$ are isomorphisms. Hence, the map
	\[
		\Q \otimes_{\Q[H]} I_{\Q[H]} \to \Q \otimes_{\Q[G]} I_{\Q[G]}
	\]
	is an isomorphism, and the map
	\[
		\Tor_1^{\Q[G]}(\Q, \Q[G] \otimes_{\Q[H]} I_{\Q[H]}) \to \Tor_2^{\Q[G]}(\Q, \Q)
	\]
	is surjective by \cref{claim: fact}. Thus, we conclude from sequence \labelcref{eqtn: long_exact_seq} that 
	\[
		\dim_{\Q} \left( \Q \otimes_{\Q[G]} \Tor_1^{\Q[H]}(\Q[G], \Q) \right) = 0.
	\] 
	Moreover, both $H$ and $G$ are groups of finite type, so $\Tor_1^{\Q[H]}(\Q[G], \Q)$ is a finitely presented $\Q[G]$-module by \cite[Proposition 1.4]{Bieri_Book} (applied on sequence \labelcref{eqtn: augmentation}).  Now, $G$ is a Lewin group, and hence $\mathcal{D}_G$ is the universal division $\Q[G]$-ring of fractions (see \cref{sec: L2Betti}). Thus, we get that
	\[
		\dim_{\mathcal{D}_G} \left( \mathcal{D}_G \otimes_{\Q[G]} \Tor_1^{\Q[H]}(\Q[G], \Q) \right) = 0,
	\]
	and so, extending scalars to $\mathcal{D}_G$ on sequence \labelcref{eqtn: augmentation} we have that
	\[
		0 \to \mathcal{D}_G \otimes_{\Q[H]} I_{\Q[H]} \to \mathcal{D}_G \otimes_{\Q[G]} I_{\Q[G]} \to 0.
	\]
	In particular,
	\[
		\dim_{\mathcal{D}_G}(\mathcal{D}_G \otimes_{\Q[H]} I_{\Q[H]}) = \dim_{\mathcal{D}_G}(\mathcal{D}_G \otimes_{\Q[G]} I_{\Q[G]}) = 1 + b_1^{(2)}(G) = 1
	\]
	since $b_1^{(2)}(G)$ vanishes by \cite[Lemma 5.4(2)]{FriedlLuck_euler}. Finally, by changing scalars to $\mathcal{D}_G$ on the augmentation map of $H$ we get
	\[
		0 \to \Tor_1^{\Q[H]}(\mathcal{D}_G, \Q) \to \mathcal{D}_G \otimes_{\Q[H]} I_{\Q[H]} \to \mathcal{D}_G \to \mathcal{D}_G \otimes_{\Q[H]} \Q \to 0.
	\]
	Therefore, since $\mathcal{D}_G \otimes_{\Q[H]} \Q$ is trivial, we conclude that 
	\[
		b_1^{\mathcal{D}_G}(H) = \dim_{\mathcal{D}_G} \left( \Tor_1^{\Q[H]}(\mathcal{D}_G, \Q)\right) = 0
	\]
	as we wanted.
\end{proof}

The following result is a combination of \cite[Theorems 4.1 \& 5.5]{FriedlLuck_euler} using equalities \labelcref{eqtn: twisted_L2_Euler,eqtn: Thurston_beta1} and the fact that every chain of group epimorphisms of the form $\Z \to G \to \Z$ is actually a chain of isomorphisms.

\begin{thm} \label{thm: L2Betti_quotient}
	Let $M$ be an admissible $3$-manifold and denote by $H$ its fundamental group. Let $G$ be a torsion-free group satisfying the Strong Atiyah Conjecture. Consider two group epimorphisms $\mu \colon H \to G$ and $\phi_1 \colon G \to \Z$. Denote by $U_1$ and $U_2$ the kernel of $\phi_1$ and $\phi_2 := \phi_1 \circ \mu$, respectively. Assume that $U_1$ is non-trivial and that
	\[
		b_1^{\mathcal{D}_G}(H) = 0.
	\]
	Then $b_1^{\mathcal{D}_{U_1}}(U_2)$ is finite and 
	\[
		b_1^{\mathcal{D}_{U_1}}(U_2) \leq b_1^{(2)}(U_2).
	\]
\end{thm}

We are now ready to show the main result.

\begin{proof}[Proof of \cref{thm: ineq_thurston3}]
	We follow the proof of \cite[Theorem 7.4]{FriedlLuck_euler}. Since seminorms are continuous and homogeneous it suffices to prove the statement for all primitive classes $\phi \in H^1(N;\Z)$. In particular, both $\phi$ and $f^{*} \phi$ are surjective. Denote by $H$ and $G$ the fundamental group of $M$ and $N$, respectively, and by $U_1$ and $U_2$ the kernel of $\phi$ and $f^{*} \phi$, respectively. By equality \labelcref{eqtn: Thurston_beta1}
	\[
		x_N(\phi) = b_1^{(2)}(U_1) \quad \mbox{and} \quad x_M(f^{*} \phi) = b_1^{(2)}(U_2).
	\]
	Thus, it suffices to show that
	\[
		 b_1^{(2)}(U_1) \leq b_1^{(2)}(U_2).
	\]
	We assume then that $U_1$ is non-trivial. Consider the natural surjection
	\[
		\Q[U_1] \otimes_{\Q[U_2]} I_{\Q[U_2]} \twoheadrightarrow I_{\Q[U_1]}.
	\]
	Taking dimensions over $\mathcal{D}_{U_1}$ yields
	\begin{equation} \label{eqn: phi_Tor1}
		b_1^{(2)}(U_1) \leq \dim_{\mathcal{D}_{U_1}}\left( \Tor_1^{\Q[U_2]}(\mathcal{D}_{U_1}, \Q)\right) = b_1^{\mathcal{D}_{U_1}}(U_2).
	\end{equation}
	Now, by \cref{thm: L2Euler_inequality} $b_1^{\mathcal{D}_G}(H)$ vanishes. Hence, from \cref{thm: L2Betti_quotient} the right term of inequality \labelcref{eqn: phi_Tor1} is finite, and moreover,
	\begin{equation} \label{eqn: Tor1_fphi}
		b_1^{\mathcal{D}_{U_1}}(U_2) \leq b_1^{(2)}(U_2).
	\end{equation}
	Therefore, combining inequalities \labelcref{eqn: phi_Tor1,eqn: Tor1_fphi}, we conclude the proof.
\end{proof}

We end this article with the proof of  \cref{cor: Thurston_ball} and its applications. In \cite{Gabai_Foliations} Gabai showed that if $f \colon M \to N$ is a proper map of positive degree $d$ between two compact and orientable $3$-manifolds $M$ and $N$, then for every $\phi \in H^1(N;\R)$ one has $d \cdot x_M(f^{*} \phi) = x_N(\phi)$. In particular, under the assumptions of \cref{thm: ineq_thurston3} we get that the number of faces of the Thurston norm balls $B_{x_M}$ and $B_{x_N}$ agree. For a general map $f$, we give uniform bounds for the number of faces of $B_{x_N}$.

\begin{proof}[Proof of \cref{cor: Thurston_ball}]
	Thurston \cite[Theorem 2, p. 106 and first paragraph, p. 107]{Thurston_Norm} showed that
	for a connected, orientable and compact $3$-manifold $Y$ the dual polytope $T(Y)^{*}$ in $H_1(Y;\R)$ is integral. Moreover, according to \cite[Lemma 4.34 (5)]{FriedlLuck_l2torsion} for every character $\phi \in H^1(Y; \R)$ we have
	\[
		x_Y(\phi) = \max_{p \in T(Y)^{*}} \phi(p).
	\] 
	By \cref{thm: ineq_thurston3} for every $\phi \in H^1(N; \R)$
	\[
		x_M(f^{*} \phi ) \geq x_N(\phi)
	\]
	where $f^{*} \phi$ is the pullback of $\phi$ with $f$, and hence
	\begin{equation} \label{eqtn: thickness}
		\max_{p \in T(N)^{*}} \phi(p) \leq \max_{p \in T(M)^{*}} \phi \circ f_1 (p) = \max_{q \in f_{1}(T(M)^{*})} \phi(q).
	\end{equation}
	We claim that $T(N)^{*} \subseteq f_{1}(T(M)^{*})$. Indeed, assume $p \notin f_{1}(T(M)^{*})$; then, since $f_{1}(T(M)^{*})$ is closed and convex, the (finite dimensional) hyperplane separation theorem (see, for example, \cite[Theorem 7]{Tik_separation}) provides $\psi \in H^1(N;\R)$ and a real number $a$  such that 
	\[
	\psi(p) > a \quad \mbox{and} \quad \psi(q)< a \quad \mbox{for all $q \in f_{1}(T(M)^{*})$}.
	\]
	Hence, by inequality (\ref{eqtn: thickness}) $p \notin T(N)^{*}$ proving the claim. 
	
	Recall that by duality the top dimensional faces of $B_{x_N}$ are in correspondence with the vertices of $T(N)^{*}$. But $T(N)^{*}$ is an integral polytope contained in $f_{1}(T(M)^{*})$, so the number of vertices is at most the number of integral points contained in $f_{1}(T(M)^{*})$.  Finally, observe that $f_1$ is induced by the natural map between $H_1(M)_{\text{f}}$ and $H_1(N)_{\text{f}}$, so $f_{1}(T(M)^{*})$ and $T(M)^{*}$ contain exactly the same number of integral points, namely $F_{x_M}$.
\end{proof}

\begin{rem}
	Note that the bound on the number of vertices of the dual Thurston polytope also implies that the numbers of lower dimensional faces of the Thurston norm ball are bounded as well.
\end{rem}

Beyond its independent interest, we hope this result serves to give upper bounds on the number of faces of the Thurston norm ball for general subclasses of admissible $3$-manifolds. For instance, in \cite{Cigna_2bridge} it was shown that the Thurston norm ball of the complement of a $2$-bridge link has at most $8$ faces. In \cite[Example 8.6]{BridsonReid_link} Bridson and Reid exhibited a link with fundamental group $H$ that surjects onto the free group of rank $2$, and hence onto the fundamental group of a $2$-bridge link, and has $b_1(H) = 2$. Thus, \cref{cor: Thurston_ball} captures the uniform bound phenomenon.

In a different direction, given a compact and orientable $3$-manifold $M$, Thurston \cite{Thurston_Norm} showed that if $\phi \in H^1(M; \R)$ is an integral character induced by a fibration of $M$ over the circle, then $\phi$ lies in the cone of an open maximal face of $B_{x_M}$, and every other integral character lying in the same cone also comes from a fibration (thus it is consistent to speak about fibered faces of $B_{x_M}$). A year later, Bieri--Neumann--Strebel \cite{BNSinv87} introduced the geometric BNS-invariant $\Sigma(G)$ of a finitely generated group $G$. They proved that $\Sigma(\pi_1(M)) \subseteq H^1(M; \R) \setminus \{0\}$ coincides with the union of the cones of the fibered faces (with the origin removed). It is a classical result (see, for instance, \cite[Proposition A4.5]{Strebel_BNSnotes}) that if $H$ is a finitely generated group that maps onto $G$ with finitely generated kernel, then $\Sigma(G)$ is determined by the intersection of $\Sigma(H)$ with $H^1(G; \R)$. Hence, if $\Sigma(H)$ is further determined by an integral polytope (see the introduction of \cite{KielakBNSviaNewton} for a detailed list of examples), the number of connected components of $\Sigma(G)$ is bounded by an integer which only depends on the polytope of $H$. In this fashion, \cref{cor: Thurston_ball} shows that under these circumstances it is still possible to give uniform upper bounds even if the kernel is not finitely generated.

\begin{cor} \label{cor: sigma}
	Let $f \colon M \to N$ be a map between admissible $3$-manifolds which induces an epimorphism on the fundamental groups and an isomorphism $f_{*} \colon H_n(M; \Q) \to H_n(N; \Q)$ for every non-negative integer $n$. Then the number of connected components of $\Sigma(\pi_1(N))$ is bounded by $F_{x_M}$.
\end{cor} 

We record that in general one cannot provide such uniform bounds. In \cite{FriedlTillmann_BNSor} Friedl and Tillmann revisited Brown's algorithm \cite{Brown_BNS} to compute the BNS-invariant of certain two-generator one-relator groups. Their algorithm (the marked polytope is a group invariant cf. \cite{HennekeKielak_agrarian}) shows that the number of connected components of the BNS-invariant of a two-generator one-relator group can be arbitrarily large, and yet the free group on two generators maps onto all of them. We thus pose the following question.

\begin{q}
	Let $H$ and $G$ be two groups such that there is an epimorphism from $H$ onto $G$, and $\Sigma(H)$ has finitely many connected components. Under which assumptions the number of connected components of $\Sigma(G)$ is bounded above by a constant depending only on $H$?
\end{q}

\bibliographystyle{alpha}
\bibliography{bib}

@article {AgolLiu_Simon,
	AUTHOR = {Agol, Ian and Liu, Yi},
	TITLE = {Presentation length and {S}imon's conjecture},
	JOURNAL = {J. Amer. Math. Soc.},
	FJOURNAL = {Journal of the American Mathematical Society},
	VOLUME = {25},
	YEAR = {2012},
	NUMBER = {1},
	PAGES = {151--187},
	ISSN = {0894-0347,1088-6834},
	MRCLASS = {57M25 (57N10)},
	MRNUMBER = {2833481},
	MRREVIEWER = {Bruno\ P.\ Zimmermann},
	DOI = {10.1090/S0894-0347-2011-00711-X},
	URL = {https://doi.org/10.1090/S0894-0347-2011-00711-X},
}

@article {BartelsFarrellLuck_FJlattices,
	AUTHOR = {Bartels, A. and Farrell, F. T. and L\"uck, W.},
	TITLE = {The {F}arrell-{J}ones conjecture for cocompact lattices in
	virtually connected {L}ie groups},
	JOURNAL = {J. Amer. Math. Soc.},
	FJOURNAL = {Journal of the American Mathematical Society},
	VOLUME = {27},
	YEAR = {2014},
	NUMBER = {2},
	PAGES = {339--388},
	ISSN = {0894-0347,1088-6834},
	MRCLASS = {18F25 (19D50 19G24 22E40 55R40 57N99)},
	MRNUMBER = {3164984},
	MRREVIEWER = {Tyrone\ Crisp},
	DOI = {10.1090/S0894-0347-2014-00782-7},
	URL = {https://doi.org/10.1090/S0894-0347-2014-00782-7},
}

@article {Berberian82,
    AUTHOR = {Berberian, Sterling Khazaz},
     TITLE = {The maximal ring of quotients of a finite von {N}eumann algebra},
   JOURNAL = {Rocky Mountain J. Math.},
  FJOURNAL = {The Rocky Mountain Journal of Mathematics},
    VOLUME = {12},
      YEAR = {1982},
    NUMBER = {1},
     PAGES = {149--164},
      ISSN = {},
   MRCLASS = {16A08 (16A30 46L10)},
  MRNUMBER = {0649748},
MRREVIEWER = {David Handelman},
       DOI = {10.1216/RMJ-1982-12-1-149},
       URL = {https://projecteuclid.org/journals/rocky-mountain-journal-of-mathematics/volume-12/issue-1/The-maximal-ring-of-quotients-of-a-finite-Von-Neumann/10.1216/RMJ-1982-12-1-149.full},
}

@Book{Bieri_Book,
	author    = {Bieri, Robert},
	publisher = {Queen Mary College, Department of Pure Mathematics, London},
	title     = {Homological dimension of discrete groups.},
	year      = {1981},
	edition   = {Second},
	mrclass   = {20J05 (18G20 57P10)},
	mrnumber  = {715779},
	pages     = {iv+198},
}

@article {BNSinv87,
    AUTHOR = {Bieri, Robert and Neumann, Walter D. and Strebel, Ralph},
     TITLE = {A geometric invariant of discrete groups},
   JOURNAL = {Invent. Math.},
  FJOURNAL = {Inventiones Mathematicae},
    VOLUME = {90},
      YEAR = {1987},
    NUMBER = {3},
     PAGES = {451--477},
      ISSN = {0020-9910},
   MRCLASS = {20J05 (22E40 57R19)},
  MRNUMBER = {914846},
MRREVIEWER = {Peter H. Kropholler},
       DOI = {10.1007/BF01389175},
       URL = {https://doi.org/10.1007/BF01389175},
}

@article{BoileauKitanoNozaki_Genera, 
	title={On the genera of symmetric unions of knots}, 
	DOI={10.4153/S0008414X25101740}, 
	journal={Canadian Journal of Mathematics}, 
	author={Boileau, Michel and Kitano, Teruaki and Nozaki, Yuta}, 
	year={2025}, 
	pages={1–26},
	}

@article {BRW_orderable3gps,
	AUTHOR = {Boyer, Steven and Rolfsen, Dale and Wiest, Bert},
	TITLE = {Orderable 3-manifold groups},
	JOURNAL = {Ann. Inst. Fourier (Grenoble)},
	FJOURNAL = {Universit\'e{} de Grenoble. Annales de l'Institut Fourier},
	VOLUME = {55},
	YEAR = {2005},
	NUMBER = {1},
	PAGES = {243--288},
	ISSN = {0373-0956,1777-5310},
	MRCLASS = {57M05 (06F15 20F34 20F60 57M50)},
	MRNUMBER = {2141698},
	MRREVIEWER = {Stephen\ P.\ Humphries},
	DOI = {10.5802/aif.2098},
	URL = {https://doi.org/10.5802/aif.2098},
}

@article {BridsonReid_link,
	AUTHOR = {Bridson, Martin R. and Reid, Alan W.},
	TITLE = {Nilpotent completions of groups, {G}rothendieck pairs, and
	four problems of {B}aumslag},
	JOURNAL = {Int. Math. Res. Not. IMRN},
	FJOURNAL = {International Mathematics Research Notices. IMRN},
	YEAR = {2015},
	NUMBER = {8},
	PAGES = {2111--2140},
	ISSN = {1073-7928,1687-0247},
	MRCLASS = {20E26 (20F18)},
	MRNUMBER = {3344664},
	MRREVIEWER = {Francesco\ de Giovanni},
	DOI = {10.1093/imrn/rnt353},
	URL = {https://doi.org/10.1093/imrn/rnt353},
}

@article {Brown_BNS,
	AUTHOR = {Brown, Kenneth S.},
	TITLE = {Trees, valuations, and the {B}ieri-{N}eumann-{S}trebel
	invariant},
	JOURNAL = {Invent. Math.},
	FJOURNAL = {Inventiones Mathematicae},
	VOLUME = {90},
	YEAR = {1987},
	NUMBER = {3},
	PAGES = {479--504},
	ISSN = {0020-9910,1432-1297},
	MRCLASS = {20F05 (20E06 22E40)},
	MRNUMBER = {914847},
	MRREVIEWER = {D.\ J.\ Collins},
	DOI = {10.1007/BF01389176},
	URL = {https://doi.org/10.1007/BF01389176},
}

@book {Cav_Good,
	AUTHOR = {Cavendish, Will},
	TITLE = {Finite-{S}heeted {C}overing {S}paces and {S}olenoids over
	3-manifolds},
	NOTE = {Thesis (Ph.D.)--Princeton University},
	PUBLISHER = {ProQuest LLC, Ann Arbor, MI},
	YEAR = {2012},
	PAGES = {96},
	ISBN = {978-1267-54485-8},
	MRCLASS = {99-05},
	MRNUMBER = {3078440},
	URL =
	{http://gateway.proquest.com/openurl?url_ver=Z39.88-2004&rft_val_fmt=info:ofi/fmt:kev:mtx:dissertation&res_dat=xri:pqm&rft_dat=xri:pqdiss:3522387},
}

@article {Cigna_2bridge,
	AUTHOR = {Cigna, Alessandro V.},
	TITLE = {The {T}hurston norm of 2-bridge link complements},
	JOURNAL = {J. Knot Theory Ramifications},
	FJOURNAL = {Journal of Knot Theory and its Ramifications},
	VOLUME = {34},
	YEAR = {2025},
	NUMBER = {12},
	PAGES = {Paper No. 2550059, 34},
	ISSN = {0218-2165,1793-6527},
	MRCLASS = {57K10 (57K31)},
	MRNUMBER = {4966656},
	DOI = {10.1142/S0218216525500592},
	URL = {https://doi.org/10.1142/S0218216525500592},
}

@article {Cohn_Universality,
	AUTHOR = {Cohn, P. M.},
	TITLE = {The embedding of firs in skew fields},
	JOURNAL = {Proc. London Math. Soc. (3)},
	FJOURNAL = {Proceedings of the London Mathematical Society. Third Series},
	VOLUME = {23},
	YEAR = {1971},
	PAGES = {193--213},
	ISSN = {0024-6115,1460-244X},
	MRCLASS = {16A06},
	MRNUMBER = {297814},
	MRREVIEWER = {G.\ Renault},
	DOI = {10.1112/plms/s3-23.2.193},
	URL = {https://doi.org/10.1112/plms/s3-23.2.193},
}

@book{cohn06FIR,
  title={Free Ideal Rings and Localization in General Rings},
  author={Cohn, Paul M.},
  isbn={9780511542794},
  lccn={},
  series={New Mathematical Monographs (3)},
  url={https://doi.org/10.1017/CBO9780511542794},
  year={2006},
  publisher={Cambridge University Press}
}

@article {ElekSzabo04,
	AUTHOR = {Elek, G\'abor and Szab\'o, Endre},
	TITLE = {Sofic groups and direct finiteness},
	JOURNAL = {J. Algebra},
	FJOURNAL = {Journal of Algebra},
	VOLUME = {280},
	YEAR = {2004},
	NUMBER = {2},
	PAGES = {426--434},
	ISSN = {0021-8693,1090-266X},
	MRCLASS = {16S34 (16E50 20C07)},
	MRNUMBER = {2089244},
	MRREVIEWER = {Pere\ Ara},
	DOI = {10.1016/j.jalgebra.2004.06.023},
	URL = {https://doi.org/10.1016/j.jalgebra.2004.06.023},
}

@article {FriedlTillmann_BNSor,
	AUTHOR = {Friedl, Stefan and Tillmann, Stephan},
	TITLE = {Two-generator one-relator groups and marked polytopes},
	JOURNAL = {Ann. Inst. Fourier (Grenoble)},
	FJOURNAL = {Universit\'e{} de Grenoble. Annales de l'Institut Fourier},
	VOLUME = {70},
	YEAR = {2020},
	NUMBER = {2},
	PAGES = {831--879},
	ISSN = {0373-0956,1777-5310},
	MRCLASS = {20J05 (20F65 22E40 57M07)},
	MRNUMBER = {4105952},
	MRREVIEWER = {Robert\ Fitzgerald\ Morse},
	DOI = {10.5802/aif.3325},
	URL = {https://doi.org/10.5802/aif.3325},
}

@article {FriedlLuck_l2torsion,
    AUTHOR = {Friedl, Stefan and L\"uck, Wolfgang},
     TITLE = {Universal {$L^2$}-torsion, polytopes and applications to
              3-manifolds},
   JOURNAL = {Proc. Lond. Math. Soc. (3)},
  FJOURNAL = {Proceedings of the London Mathematical Society. Third Series},
    VOLUME = {114},
      YEAR = {2017},
    NUMBER = {6},
     PAGES = {1114--1151},
      ISSN = {0024-6115,1460-244X},
   MRCLASS = {57Q10 (19B28 20J05 57M27)},
  MRNUMBER = {3661347},
MRREVIEWER = {Fumihiro\ Ushitaki},
       DOI = {10.1112/plms.12035},
       URL = {https://doi.org/10.1112/plms.12035},
}

@article {FriedlLuck_euler,
    AUTHOR = {Friedl, Stefan and L\"{u}ck, Wolfgang},
     TITLE = {{$L^2$}-{E}uler characteristics and the {T}hurston norm},
   JOURNAL = {Proc. Lond. Math. Soc. (3)},
  FJOURNAL = {Proceedings of the London Mathematical Society. Third Series},
    VOLUME = {118},
      YEAR = {2019},
    NUMBER = {4},
     PAGES = {857--900},
      ISSN = {0024-6115},
   MRCLASS = {57M27 (22D25 58J52)},
  MRNUMBER = {3938714},
MRREVIEWER = {Nikhil Savale},
       DOI = {10.1112/plms.12202},
       URL = {https://doi-org.ezproxy-prd.bodleian.ox.ac.uk/10.1112/plms.12202},
}

@article {FriedlSchreveTillmann_ThurstonFox,
    AUTHOR = {Friedl, Stefan and Schreve, Kevin and Tillmann, Stephan},
     TITLE = {Thurston norm via {F}ox calculus},
   JOURNAL = {Geom. Topol.},
  FJOURNAL = {Geometry \& Topology},
    VOLUME = {21},
      YEAR = {2017},
    NUMBER = {6},
     PAGES = {3759--3784},
      ISSN = {1465-3060},
   MRCLASS = {57M27 (20J05 57M05 57R19)},
  MRNUMBER = {3693575},
MRREVIEWER = {Daniel Matei},
       DOI = {10.2140/gt.2017.21.3759},
       URL = {https://doi-org.ezproxy-prd.bodleian.ox.ac.uk/10.2140/gt.2017.21.3759},
}

@article {Gabai_Foliations,
	AUTHOR = {Gabai, David},
	TITLE = {Foliations and the topology of {$3$}-manifolds},
	JOURNAL = {J. Differential Geom.},
	FJOURNAL = {Journal of Differential Geometry},
	VOLUME = {18},
	YEAR = {1983},
	NUMBER = {3},
	PAGES = {445--503},
	ISSN = {0022-040X,1945-743X},
	MRCLASS = {57N10 (57R30)},
	MRNUMBER = {723813},
	MRREVIEWER = {Jean-Pierre\ Otal},
	URL = {http://projecteuclid.org/euclid.jdg/1214437784},
}

@article {Gabai_FoliationsIII,
	AUTHOR = {Gabai, David},
	TITLE = {Foliations and the topology of {$3$}-manifolds. {III}},
	JOURNAL = {J. Differential Geom.},
	FJOURNAL = {Journal of Differential Geometry},
	VOLUME = {26},
	YEAR = {1987},
	NUMBER = {3},
	PAGES = {479--536},
	ISSN = {0022-040X,1945-743X},
	MRCLASS = {57N10 (57R30)},
	MRNUMBER = {910018},
	MRREVIEWER = {Jean-Pierre\ Otal},
	URL = {http://projecteuclid.org/euclid.jdg/1214441488},
}

@article {Grater20,
    AUTHOR = {Gr{\"a}ter, Joachim},
     TITLE = {Free division rings of fractions of crossed products of groups
              with {C}onradian left-orders},
   JOURNAL = {Forum Math.},
  FJOURNAL = {Forum Mathematicum},
    VOLUME = {32},
      YEAR = {2020},
    NUMBER = {3},
     PAGES = {739--772},
      ISSN = {0933-7741},
   MRCLASS = {16S35 (12E15 16S34 16S85 16W60 20F60)},
  MRNUMBER = {4095506},
       DOI = {10.1515/forum-2019-0264},
       URL = {https://ezproxy-prd.bodleian.ox.ac.uk:2102/10.1515/forum-2019-0264},
}

@article {GordonLuecke_Knots,
	AUTHOR = {Gordon, C. McA. and Luecke, J.},
	TITLE = {Knots are determined by their complements},
	JOURNAL = {J. Amer. Math. Soc.},
	FJOURNAL = {Journal of the American Mathematical Society},
	VOLUME = {2},
	YEAR = {1989},
	NUMBER = {2},
	PAGES = {371--415},
	ISSN = {0894-0347,1088-6834},
	MRCLASS = {57M25 (57M40)},
	MRNUMBER = {965210},
	MRREVIEWER = {Martin\ Scharlemann},
	DOI = {10.2307/1990979},
	URL = {https://doi.org/10.2307/1990979},
}

@article {HartleyMurasugi_Meridians,
	AUTHOR = {Hartley, Richard and Murasugi, Kunio},
	TITLE = {Homology invariants},
	JOURNAL = {Canadian J. Math.},
	FJOURNAL = {Canadian Journal of Mathematics. Journal Canadien de
	Math\'ematiques},
	VOLUME = {30},
	YEAR = {1978},
	NUMBER = {3},
	PAGES = {655--670},
	ISSN = {0008-414X,1496-4279},
	MRCLASS = {57M25 (57M10)},
	MRNUMBER = {491590},
	MRREVIEWER = {S.\ C.\ Althoen},
	DOI = {10.4153/CJM-1978-057-6},
	URL = {https://doi.org/10.4153/CJM-1978-057-6},
}

@incollection {Hempel_RF3mflds,
	AUTHOR = {Hempel, John},
	TITLE = {Residual finiteness for {$3$}-manifolds},
	BOOKTITLE = {Combinatorial group theory and topology ({A}lta, {U}tah,
	1984)},
	SERIES = {Ann. of Math. Stud.},
	VOLUME = {111},
	PAGES = {379--396},
	PUBLISHER = {Princeton Univ. Press, Princeton, NJ},
	YEAR = {1987},
	ISBN = {0-691-08409-2; 0-691-08410-6},
	MRCLASS = {57M05 (20E26 20F34 57N10)},
	MRNUMBER = {895623},
}

@article {HennekeLopez_pseudoSylvester,
    AUTHOR = {Henneke, Fabian and L\'{o}pez-\'{A}lvarez, Diego},
     TITLE = {Pseudo-{S}ylvester domains and skew {L}aurent polynomials over
              firs},
   JOURNAL = {J. Algebra Appl.},
  FJOURNAL = {Journal of Algebra and its Applications},
    VOLUME = {21},
      YEAR = {2022},
    NUMBER = {8},
     PAGES = {Paper No. 2250168, 28},
      ISSN = {0219-4988,1793-6829},
   MRCLASS = {16E60 (16E30 16K40 19A31)},
  MRNUMBER = {4469346},
MRREVIEWER = {Ivan\ D.\ Chipchakov},
       DOI = {10.1142/S0219498822501687},
       URL = {https://doi.org/10.1142/S0219498822501687},
}

@article {HennekeKielak_agrarian,
    AUTHOR = {Henneke, Fabian and Kielak, Dawid},
     TITLE = {Agrarian and {$L^2$}-invariants},
   JOURNAL = {Fund. Math.},
  FJOURNAL = {Fundamenta Mathematicae},
    VOLUME = {255},
      YEAR = {2021},
    NUMBER = {3},
     PAGES = {255--287},
      ISSN = {0016-2736},
   MRCLASS = {20J05 (12E15 16S35 20E06 57Q10)},
  MRNUMBER = {4324826},
MRREVIEWER = {Brita E. A. Nucinkis},
       DOI = {10.4064/fm808-4-2021},
       URL = {https://doi.org/10.4064/fm808-4-2021},
}

@article {HughesDivRings1970,
    AUTHOR = {Hughes, Ian},
     TITLE = {Division rings of fractions for group rings},
   JOURNAL = {Comm. Pure Appl. Math.},
  FJOURNAL = {Communications on Pure and Applied Mathematics},
    VOLUME = {23},
      YEAR = {1970},
     PAGES = {181--188},
      ISSN = {0010-3640},
   MRCLASS = {20.80},
  MRNUMBER = {263934},
MRREVIEWER = {Sudarshan K. Sehgal},
       DOI = {10.1002/cpa.3160230205},
       URL = {https://ezproxy-prd.bodleian.ox.ac.uk:2102/10.1002/cpa.3160230205},
}

@article {JaikinZapirain2020THEUO,
    AUTHOR = {Jaikin-Zapirain, Andrei},
     TITLE = {The universality of {H}ughes-free division rings},
   JOURNAL = {Selecta Math. (N.S.)},
  FJOURNAL = {Selecta Mathematica. New Series},
    VOLUME = {27},
      YEAR = {2021},
    NUMBER = {4},
     PAGES = {Paper No. 74, 33},
      ISSN = {1022-1824},
   MRCLASS = {16K40 (12E15 16S34 16S35 20F65)},
  MRNUMBER = {4292784},
MRREVIEWER = {Mai Hoang Bien},
       DOI = {10.1007/s00029-021-00691-w},
       URL = {https://ezproxy-prd.bodleian.ox.ac.uk:2102/10.1007/s00029-021-00691-w},
}

@article {JaikinLinton_oneRelCoherence,
	AUTHOR = {Jaikin-Zapirain, Andrei and Linton, Marco},
	TITLE = {On the coherence of one-relator groups and their group
	algebras},
	JOURNAL = {Ann. of Math. (2)},
	FJOURNAL = {Annals of Mathematics. Second Series},
	VOLUME = {201},
	YEAR = {2025},
	NUMBER = {3},
	PAGES = {909--959},
	ISSN = {0003-486X,1939-8980},
	MRCLASS = {20E07 (20E08 20J05)},
	MRNUMBER = {4899802},
	MRREVIEWER = {Stefan\ Kohl},
	DOI = {10.4007/annals.2025.201.3.4},
	URL = {https://doi.org/10.4007/annals.2025.201.3.4},
}

@article {JaikinLopez_Atiyah,
    AUTHOR = {Jaikin-Zapirain, Andrei and L\'{o}pez-{\'{A}}lvarez, Diego},
     TITLE = {The strong {A}tiyah and {L}\"{u}ck approximation conjectures for
              one-relator groups},
   JOURNAL = {Math. Ann.},
  FJOURNAL = {Mathematische Annalen},
    VOLUME = {376},
      YEAR = {2020},
    NUMBER = {3-4},
     PAGES = {1741--1793},
      ISSN = {0025-5831},
   MRCLASS = {20F05 (16K40 16S34)},
  MRNUMBER = {4081128},
MRREVIEWER = {Shoumin Liu},
       DOI = {10.1007/s00208-019-01926-0},
       URL = {https://doi.org/10.1007/s00208-019-01926-0},
}

@article {JaikinSouza_Sylvesterprop,
	AUTHOR = {Jaikin-Zapirain, Andrei and Souza, Henrique},
	TITLE = {Sylvester domains and pro-{$p$} groups},
	JOURNAL = {Doc. Math.},
	FJOURNAL = {Documenta Mathematica},
	VOLUME = {31},
	YEAR = {2026},
	NUMBER = {2},
	PAGES = {453--502},
	ISSN = {1431-0635,1431-0643},
	MRCLASS = {20E18 (16S35 16S85 20E22)},
	MRNUMBER = {5009504},
	DOI = {10.4171/dm/1034},
	URL = {https://doi.org/10.4171/dm/1034},
}

@book {Kaplansky72,
	AUTHOR = {Kaplansky, Irving},
	TITLE = {Fields and rings},
	SERIES = {Chicago Lectures in Mathematics},
	EDITION = {Second},
	PUBLISHER = {University of Chicago Press, Chicago, Ill.-London},
	YEAR = {1972},
	PAGES = {x+206},
	MRCLASS = {13-02 (12-02 16-02)},
	MRNUMBER = {349646},
}

@article {KielakBNSviaNewton,
    AUTHOR = {Kielak, Dawid},
     TITLE = {The {B}ieri-{N}eumann-{S}trebel invariants via {N}ewton
              polytopes},
   JOURNAL = {Invent. Math.},
  FJOURNAL = {Inventiones Mathematicae},
    VOLUME = {219},
      YEAR = {2020},
    NUMBER = {3},
     PAGES = {1009--1068},
      ISSN = {0020-9910},
   MRCLASS = {20F65},
  MRNUMBER = {4055183},
MRREVIEWER = {Igor Belegradek},
       DOI = {10.1007/s00222-019-00919-9},
       URL = {https://doi.org/10.1007/s00222-019-00919-9},
}

@article {KielakLinton_3mfldAtiyah,
    AUTHOR = {Kielak, Dawid and Linton, Marco},
     TITLE = {Group rings of three-manifold groups},
   JOURNAL = {Proc. Amer. Math. Soc.},
  FJOURNAL = {Proceedings of the American Mathematical Society},
    VOLUME = {152},
      YEAR = {2024},
    NUMBER = {5},
     PAGES = {1939--1946},
      ISSN = {0002-9939,1088-6826},
   MRCLASS = {20J05 (20F34 57K30)},
  MRNUMBER = {4728464},
       DOI = {10.1090/proc/16716},
       URL = {https://doi.org/10.1090/proc/16716},
}

@incollection {Kirby_Problems,
	AUTHOR = {Kirby, Rob},
	TITLE = {Problems in low dimensional manifold theory},
	BOOKTITLE = {Algebraic and geometric topology ({P}roc. {S}ympos. {P}ure
	{M}ath., {S}tanford {U}niv., {S}tanford, {C}alif., 1976),
	{P}art 2},
	SERIES = {Proc. Sympos. Pure Math.},
	VOLUME = {XXXII},
	PAGES = {273--312},
	PUBLISHER = {Amer. Math. Soc., Providence, RI},
	YEAR = {1978},
	ISBN = {0-8218-1433-8},
	MRCLASS = {57-02},
	MRNUMBER = {520548},
	MRREVIEWER = {Louis\ H.\ Kauffman},
}

@article {LinnellDivRings93,
    AUTHOR = {Linnell, Peter A.},
     TITLE = {Division rings and group von {N}eumann algebras},
   JOURNAL = {Forum Math.},
  FJOURNAL = {Forum Mathematicum},
    VOLUME = {5},
      YEAR = {1993},
    NUMBER = {6},
     PAGES = {561--576},
      ISSN = {0933-7741},
   MRCLASS = {20C07 (16K40 16S35 22D25 46L10)},
  MRNUMBER = {1242889},
MRREVIEWER = {Alain Valette},
       DOI = {10.1515/form.1993.5.561},
       URL = {https://ezproxy-prd.bodleian.ox.ac.uk:2102/10.1515/form.1993.5.561},
}

@incollection {LinnellZeroDivc,
    AUTHOR = {Linnell, Peter A.},
     TITLE = {Analytic versions of the zero divisor conjecture},
 BOOKTITLE = {Geometry and cohomology in group theory ({D}urham, 1994)},
    SERIES = {London Math. Soc. Lecture Note Ser.},
    VOLUME = {252},
     PAGES = {209--248},
 PUBLISHER = {Cambridge Univ. Press, Cambridge},
      YEAR = {1998},
      ISBN = {0-521-63556-X},
   MRCLASS = {20C07 (43A99 46L99 46M20)},
  MRNUMBER = {1709960},
       DOI = {10.1017/CBO9780511666131.015},
       URL = {https://doi.org/10.1017/CBO9780511666131.015},
}

@incollection {Linnell_localizations,
	AUTHOR = {Linnell, Peter A.},
	TITLE = {Noncommutative localization in group rings},
	BOOKTITLE = {Non-commutative localization in algebra and topology},
	SERIES = {London Math. Soc. Lecture Note Ser.},
	VOLUME = {330},
	PAGES = {40--59},
	PUBLISHER = {Cambridge Univ. Press, Cambridge},
	YEAR = {2006},
	ISBN = {978-0-521-68160-5; 0-521-68160-X},
	MRCLASS = {16S34 (16S10 20C07)},
	MRNUMBER = {2222481},
	MRREVIEWER = {Mykola\ Ya.\ Komarnitski\u i},
	DOI = {10.1017/CBO9780511526381.010},
	URL = {https://doi.org/10.1017/CBO9780511526381.010},
}

@book {Luck02,
    AUTHOR = {L{\"u}ck, Wolfgang},
     TITLE = {{$L\sp 2$}-invariants: theory and applications to geometry and
              {$K$}-theory},
  PUBLISHER = {Springer-Verlag},
   ADDRESS = {Berlin},
      YEAR = {2002},
     PAGES = {xvi+595},
      ISBN = {3-540-43566-2},
   MRCLASS = {58J22 (19K56 46L80 57Q10 57R20 58J52)},
  MRNUMBER = {2003m:58033},
MRREVIEWER = {Thomas Schick},
}

@book {Rotman09,
    AUTHOR = {Rotman, Joseph J.},
     TITLE = {An introduction to homological algebra},
  PUBLISHER = {Springer},
   ADDRESS = {New York},
      YEAR = {2009},
     PAGES = {xiv+709},
      ISBN = {978-0-387-24527-0},
   MRCLASS = {18Gxx (13Dxx 16Exx 18-01 20J06)},
  MRNUMBER = {2455920},
MRREVIEWER = {Fernando Muro},
}

@article {Roushon_FJ3mfld,
	AUTHOR = {Roushon, S. K.},
	TITLE = {The {F}arrell-{J}ones isomorphism conjecture for 3-manifold
	groups},
	JOURNAL = {J. K-Theory},
	FJOURNAL = {Journal of K-Theory. K-Theory and its Applications in Algebra,
	Geometry, Analysis \& Topology},
	VOLUME = {1},
	YEAR = {2008},
	NUMBER = {1},
	PAGES = {49--82},
	ISSN = {1865-2433,1865-5394},
	MRCLASS = {57N37 (19D35 19J10)},
	MRNUMBER = {2424566},
	DOI = {10.1017/is007011012jkt005},
	URL = {https://doi.org/10.1017/is007011012jkt005},
}

@article {SchickIntegrality,
    AUTHOR = {Schick, Thomas},
     TITLE = {Erratum: ``{I}ntegrality of {$L^2$}-{B}etti numbers''},
   JOURNAL = {Math. Ann.},
  FJOURNAL = {Mathematische Annalen},
    VOLUME = {322},
      YEAR = {2002},
    NUMBER = {2},
     PAGES = {421--422},
      ISSN = {0025-5831},
   MRCLASS = {55N25 (16S34 46L99 58J22)},
  MRNUMBER = {1894160},
MRREVIEWER = {Warren Dicks},
       DOI = {10.1007/s002080100282},
       URL = {https://doi.org/10.1007/s002080100282},
}

@article {Schreve_AtiyahVCS,
    AUTHOR = {Schreve, Kevin},
     TITLE = {The strong {A}tiyah conjecture for virtually cocompact special
              groups},
   JOURNAL = {Math. Ann.},
  FJOURNAL = {Mathematische Annalen},
    VOLUME = {359},
      YEAR = {2014},
    NUMBER = {3-4},
     PAGES = {629--636},
      ISSN = {0025-5831},
   MRCLASS = {20F65},
  MRNUMBER = {3231009},
MRREVIEWER = {Qin Wang},
       DOI = {10.1007/s00208-014-1007-9},
       URL = {https://doi.org/10.1007/s00208-014-1007-9},
}

@article {Scott_core,
    AUTHOR = {Scott, G. Peter},
     TITLE = {Compact submanifolds of {$3$}-manifolds},
   JOURNAL = {J. London Math. Soc. (2)},
  FJOURNAL = {Journal of the London Mathematical Society. Second Series},
    VOLUME = {7},
      YEAR = {1973},
     PAGES = {246--250},
      ISSN = {0024-6107},
   MRCLASS = {57A10},
  MRNUMBER = {326737},
MRREVIEWER = {C. D. Feustel},
       DOI = {10.1112/jlms/s2-7.2.246},
       URL = {https://doi-org.ezproxy-prd.bodleian.ox.ac.uk/10.1112/jlms/s2-7.2.246},
}

@article {Serre_CohomologyFI,
    AUTHOR = {Serre, Jean-Pierre},
     TITLE = {Cohomologie des groupes discrets},
   JOURNAL = {Ann. of Math. Studies},
  FJOURNAL = {Annals of Mathematics Studies},
    VOLUME = {70},
      YEAR = {1971},
    NUMBER = {},
     PAGES = {77--169},
      ISSN = {},
   MRCLASS = {22E40},
  MRNUMBER = {0385006},
MRREVIEWER = {H. Bass},
       DOI = {},
       URL = {},
}

@misc{Strebel_BNSnotes,
	title = {Notes on the Sigma invariants}, 
	author = {Strebel, Ralph},
	howpublished = "\url{https://arxiv.org/abs/1204.0214}",
	year = {2013},
	eprint = {1204.0214},
	archivePrefix = {arXiv},
	primaryClass = {math.GR}
}

@article {Swancd1,
    AUTHOR = {Swan, Richard G.},
     TITLE = {Groups of cohomological dimension one},
   JOURNAL = {J. Algebra},
  FJOURNAL = {Journal of Algebra},
    VOLUME = {12},
      YEAR = {1969},
     PAGES = {585--610},
      ISSN = {0021-8693},
   MRCLASS = {20.10},
  MRNUMBER = {240177},
MRREVIEWER = {L.\ Neuwirth},
       DOI = {10.1016/0021-8693(69)90030-1},
       URL = {https://doi.org/10.1016/0021-8693(69)90030-1},
}

@article {Thurston_Norm,
	AUTHOR = {Thurston, William P.},
	TITLE = {A norm for the homology of {$3$}-manifolds},
	JOURNAL = {Mem. Amer. Math. Soc.},
	FJOURNAL = {Memoirs of the American Mathematical Society},
	VOLUME = {59},
	YEAR = {1986},
	NUMBER = {339},
	PAGES = {i--vi and 99--130},
	ISSN = {0065-9266,1947-6221},
	MRCLASS = {57N10 (57M25 57R20 57R30)},
	MRNUMBER = {823443},
	MRREVIEWER = {Martin\ Scharlemann},
}

@incollection {Tik_separation,
	AUTHOR = {Tikhomirov, V. M.},
	TITLE = {Convex analysis},
	BOOKTITLE = {Current problems in mathematics. {F}undamental directions,
	{V}ol.\ 14 ({R}ussian)},
	SERIES = {Itogi Nauki i Tekhniki},
	PAGES = {5--101, 272},
	PUBLISHER = {Akad. Nauk SSSR, Vsesoyuz. Inst. Nauchn. i Tekhn. Inform.,
	Moscow},
	YEAR = {1987},
	MRCLASS = {49-02 (41A65 46N05 49-01 49A50 90C25)},
	MRNUMBER = {915773},
}

\end{document}